\documentclass[10pt]{amsart}

\usepackage{graphicx}

\usepackage[USenglish]{babel}
\usepackage{latexsym}
\usepackage{amsmath}
\usepackage{amsfonts}
\usepackage{amssymb}
\usepackage{esint}
\usepackage[all]{xy}
\renewcommand{\a }{\alpha }
\renewcommand{\b }{\beta }
\renewcommand{\d}{\delta }
\newcommand{\D }{\Delta }

\newcommand{\e }{\varepsilon }

\renewcommand{\l }{\lambda }

\newcommand{\n }{\nabla }
\newcommand{\var }{\varphi }
\newcommand{\rh }{\rho }

\newcommand{\s }{\sigma }

\renewcommand{\t }{\tau }
\renewcommand{\th }{\theta }

\renewcommand{\O }{\Omega }

\newcommand{\ov}{\overline}

\def\p{\partial}

\newcommand{\wtilde }{\widetilde}

\newcommand{\be}{\begin{equation}}
\newcommand{\ee}{\end{equation}}

\newcommand{\R}{\mathbb{R}}
\renewcommand{\S}{\mathbb{S}}

\newcommand{\N}{\mathbb{N}}
\newcommand{\no}{\noindent}
\newcommand{\dis}{\displaystyle}

\newcommand{\dkr}{{\bf d}}

\newtheorem{theorem}{Theorem}[section]
\newtheorem{proposition}[theorem]{Proposition}
\newtheorem{example}[theorem]{Example}

\newcommand{\bpr}{\begin{proposition}}
\newcommand{\epr}{\end{proposition}}
\newcommand{\bex}{\begin{example}\rm}
\newcommand{\eex}{\end{example}}

\begin{document}

\newtheorem{lem}{Lemma}[section]
\newtheorem{pro}[lem]{Proposition}
\newtheorem{thm}[lem]{Theorem}
\newtheorem{rem}[lem]{Remark}
\newtheorem{cor}[lem]{Corollary}
\newtheorem{df}[lem]{Definition}

\title[Mean field equation with variable intensities]{Blow-up analysis and existence results in the supercritical case for an asymmetric mean field equation with\\ variable intensities}

\author{Aleks Jevnikar}

\address{ Aleks Jevnikar,~University of Rome `Tor Vergata', Via della Ricerca Scientifica 1, 00133 Roma, Italy}
\email{jevnikar@mat.uniroma2.it}


\keywords{Geometric PDEs, Mean field equation, Blow-up analysis, Variational methods}

\subjclass[2000]{ 35J61, 35J20, 35R01, 35B44.}

\begin{abstract}
A class of equations with exponential nonlinearities on a compact Riemannian surface is considered. More precisely, we study an asymmetric sinh-Gordon problem arising as a mean field equation of the equilibrium turbulence of vortices with variable intensities.

We start by performing a blow-up analysis in order to derive some information on the local blow-up masses. As a consequence we get a compactness property in a supercritical range.

We next introduce a variational argument based on improved Moser-Trudinger inequalities which yields existence of solutions for any choice of the underlying surface.
\end{abstract}

\maketitle

\section{Introduction}

\medskip

We consider here the following equation
\begin{equation} \label{eq}
  - \D u = \rho_1 \left( \frac{h_1 e^{u}}{\int_M
      h_1 e^{u} \,dV_g} - \frac{1}{|M|} \right) - a \rho_2 \left( \frac{h_2 e^{-au}}{\int_M
      h_2 e^{-au} \,dV_g} - \frac{1}{|M|} \right),
\end{equation}
where $a\in(0,1)$, $h_1, h_2$ are smooth positive functions, $\rho_1, \rho_2$ are two positive parameters  and $(M,g)$ is a compact orientable surface with no boundary equipped with a Riemannian metric $g$. For the sake of simplicity, we normalize the total volume of $M$ so that $|M|=1$.

Equation \eqref{eq} arises in the context of the statistical mechanics description of $2D$-turbulence: the physical model was first introduced in \cite{ons} and different mean field equations have been obtained according to different constraints. In the case that the circulation number density is subject to a
probability measure, under a \emph{deterministic} assumption on the vortex intensities, the model is ruled by the following equation, see \cite{sa-su}:
\begin{equation} \label{eq:prob}
  - \D u = \rho \int_{[-1,1]} \a \left( \frac{ e^{\a u}}{\int_M
      e^{\a u} \,dV_g} - \frac{1}{|M|} \right) \mathcal P(d\a),
\end{equation}
where $u$ denotes the stream function of a turbulent Euler flow, $\mathcal P$ is a Borel probability measure defined on the interval $[-1,1]$ describing the point vortex intensity distribution and $\rho>0$ is a physical constant related to the inverse temperature. Equation \eqref{eq} is related to the latter model for the particular choice $\mathcal P(d\a)= \t_1 \d_1(d\a)+\t_a \d_{-a}(d\a)$, where $a\in(0,1)$ and $\t_1,\t_a$ are positive parameters such that $\t_1+\t_a=1$. Observe that we focus just one the different-sign problem since the case $supp\, \mathcal P \subset [0,1]$ presents some differences and it is considered in \cite{je-ya2}.

\medskip

In order to describe the nature of equation \eqref{eq} and the strategy to attack it, let us first consider the standard mean field equation obtained from \eqref{eq:prob} with $\mathcal P(d\a) = \d_1$, namely
\begin{equation} \label{eq:liouv}
  - \D u = \rho\left( \frac{h\, e^{u}}{\int_M  h\, e^{u} \,dV_g}- \frac{1}{|M|} \right).
\end{equation}
The latter equation has been widely studied since it is related to the prescribed Gaussian curvature problem \cite{ba, cha, cha2, ly, s} and to the mean field equation of Euler flows \cite{clmp, ki}. For a survey of the latter equation we refer to \cite{mal, tar}.

One of the main difficulties in dealing with this class of equations is due to the loss of compactness, as its solutions might blow-up. As a consequence, the first step is to analyze the bubbling phenomenon. We point out an important property that was observed for \eqref{eq:liouv}, see \cite{bm, li, li-sha}: for a sequence of blow-up solutions $\{u_k\}_k$ to \eqref{eq:liouv} relative to $\rho_k$ with blow-up point $\bar x$ the following quantization holds true 
\begin{equation} \label{liouv:quant}
	\tilde\s(\bar x)= \lim_{\d\to 0} \lim_{k\to+\infty} \rho_k \frac{\int_{B_\d(\bar x)}h\, e^{u_k}}{\int_M  h\, e^{u_k} \,dV_g} = 8\pi.
\end{equation}
The latter property yields important consequences in many applications, in particular for what concerns compactness results, see the discussion later on.

In the more general situation of \eqref{eq} (and \eqref{eq:prob}) the blow-up analysis has still to be completed. We refer to \cite{jwy, os, os1, ors, ri-ze} for the progress in this direction. We stress that for $a=1$ equation \eqref{eq} reduces to the sinh-Gordon problem
\begin{equation} \label{eq:sinh}
  - \D u = \rho_1 \left( \frac{h_1 e^{u}}{\int_M
      h_1 e^{u} \,dV_g} - \frac{1}{|M|} \right) - \rho_2 \left( \frac{h_2 e^{-u}}{\int_M
      h_2 e^{-u} \,dV_g} - \frac{1}{|M|} \right),
\end{equation}
which has been very much studied recently \cite{bjmr, jev, jev2, jev3, jwy2, pi-ri}. For what concerns the quantization property \eqref{liouv:quant} a similar result was derived in \cite{jwy} (see also \cite{je-ya} for a similar approach). Indeed, the authors proved that for a blow-up sequence $\{u_k\}_k$ to \eqref{eq:sinh} one has
\begin{equation} \label{mass:sinh}
	\tilde\s_1(\bar x)=\lim_{\d\to 0} \lim_{k\to+\infty} \rho_{1,k} \frac{\int_{B_\d(\bar x)}h_1 e^{u_k}}{\int_M  h_1 e^{u_k} \,dV_g} \in 8\pi\N, \qquad 	\tilde\s_2(\bar x)=\lim_{\d\to 0} \lim_{k\to+\infty} \rho_{2,k} \frac{\int_{B_\d(\bar x)}h_2 e^{-u_k}}{\int_M  h_2 e^{-u_k} \,dV_g} \in 8\pi\N.
\end{equation}
In fact, one can construct such blowing-up solutions \cite{ew, gp}. Our plan is to apply the same strategy to the general equation \eqref{eq}: in this case the quantization does not have a so simple description due to the asymmetry of the exponential terms in \eqref{eq}. Nevertheless, we are able to derive the following partial result by assuming an a priori bound. Denoting by $\mathring H^1(M) = \left\{ u\in H^1(M) \; : \; \int_M u = 0 \right\}$ we have:
\begin{theorem} \label{th:mass}
Let $\{u_k\}_k\subset \mathring H^1(M)$ be a sequence of blow-up solutions to \eqref{eq} relative to $(\rho_{1,k},\rho_{2,k})\to(\bar {\rho}_1, \bar {\rho}_2)$ with blow-up point $\bar x \in M$ and let $\s_1(\bar x),\, \s_2(\bar x)$ be the local blow-up masses relative to $u,\,-au$, respectively, defined similarly as in \eqref{mass:sinh}. Then, it holds:
\begin{enumerate}
	\item If $\bar \rho_2 < \dfrac{8\pi}{a^2}$ (resp. $\bar \rho_1 < 8\pi$), then $(\s_1(\bar x), \s_2(\bar x))$ is given by
	$$
		(8\pi, 0) \qquad \left(\mbox{resp. } \left(0,\dfrac{8\pi}{a^2}\right)  \right).
	$$
	
	\item If $\bar\rho_1 < 16\pi$, $\bar \rho_2 < \dfrac{16\pi}{a^2}$, then $(\s_1(\bar x), \s_2(\bar x))$ is one of the following types:
	$$\begin{array}{ll}
	 \mbox{if } a\geq \dfrac 12 \,: &	\qquad (8\pi, 0), \qquad \left(0,\dfrac{8\pi}{a^2}\right), \vspace{0.1cm}\\
	 \mbox{if } a< \dfrac 12 \,:&	\qquad (8\pi, 0), \qquad \left(0,\dfrac{8\pi}{a^2}\right), \qquad \left(8\pi,\dfrac{8\pi}{a^2}+\dfrac{16\pi}{a}\right).
	\end{array}$$
\end{enumerate}
\end{theorem}

\medskip

We point out that recently in \cite{ri-ze2} the authors exhibit a minimum blow-up mass for the two components as in Theorem \ref{th:mass}: $(8\pi, 0)$ and $\left(0,\dfrac{8\pi}{a^2}\right)$, respectively, and an existence result in the spirit of Theorem~\ref{th:ex1} under some assumptions on the first eigenvalue of $-\D$ is provided.

We follow here the argument in \cite{jwy} concerning the sinh-Gordon case \eqref{eq:sinh} (see also \cite{je-ya, lin-wei-zh} for the Tzitz\'eica equation and $SU(3)$ Toda system, respectively). We start by introducing a selection process to detect a finite number of disks where the local energy is related to that of globally defined Liouville equations.  In each disk the local mass of the two components $u_k$ and $-a u_k$ is quantized according to a local Pohozaev identity. We then use the bound on the parameters $\rho_i$ to exclude some configurations which may produce other contributions to the local masses when combining the blowing-up disks.  

By standard arguments the information on the local mass in Theorem \ref{th:mass} yields some compactness properties, see for example \cite{jwy2}.
\begin{theorem} \label{th:comp}
We have the following:
\begin{enumerate}
	\item Let $K$ be a compactly supported subset either of $\left(\R\setminus\{8\pi\N\}\right)\times \left(0, \dfrac{8\pi}{a^2} \right)\subset \R^2$ or $(0,8\pi)\times \left(\R\setminus\left\{\dfrac{8\pi}{a^2}\N\right\}\right) \subset \R^2$. Then, the family of solutions $\{u_\rho\}_{\rho\in K}\subset \mathring H^1(M)$ of \eqref{eq} relative to $\rho=(\rho_1,\rho_2)$ is uniformly bounded in $C^{2,\a}(M)$ for some $\a>0$.
	
	\medskip
	
	\item Let $K$ be a compactly supported subset of $(8\pi, 16\pi) \times \left(\left(\dfrac{8\pi}{a^2}, \dfrac{16\pi}{a^2} \right)\setminus \left\{ \dfrac{8\pi}{a^2}+\dfrac{16\pi}{a} \right\}\right) \subset \R^2$. Then, the same conclusion as in \emph{(1)} holds true.
\end{enumerate}
\end{theorem} 

\medskip

Let us now focus on some variational aspects concerning this class of problems. In order to understand how to handle this kind of equations, let us start with the standard mean field equation \eqref{eq:liouv}. In this case the associated energy functional is given by $I_\rho : H^1(M)\to\R$, 
\begin{equation} \label{liouv:func}
	I_\rho(u) = \frac{1}{2}\int_M |\nabla u|^2 \,dV_g - \rho  \left( \log\int_M h\,  e^u \,dV_g  - \int_M u \,dV_g \right).
\end{equation}
The basic tool in this framework is the Moser-Trudinger inequality
\begin{equation}\label{ineq}
	8\pi \log\int_M e^{u-\ov u} \, dV_g \leq \frac 12 \int_M |\n u|^2\,dV_g + C_{M,g}, \qquad \ov u= \fint_M u\,dV_g.
\end{equation}
By the latter inequality we readily deduce that $I_\rho$ is bounded from below and coercive if $\rho<8\pi$ and the global minimum corresponds to a solution of \eqref{eq:liouv}. As soon as $\rho>8\pi$ the functional $I_\rho$ is unbounded from below and the minimization technique is no more possible. A successful strategy is to introduce improved Moser-Trudinger inequalities based on the \emph{spreading} of $e^u$  over the surface \cite{ch-li}. By using this kind of inequalities one can show that if $\rho<8(k+1)\pi, k\in\N$ and $I_\rho(u)$ is large negative, $e^u$ need to concentrate around at most $k$ points of $M$. It is then natural to introduce the set of $k$-th \emph{formal barycentres} of $M$
\begin{equation}	\label{M_k}
	M_k = \left\{ \sum_{i=1}^k t_i \d_{x_i} \, : \, \sum_{i=1}^k t_i=1, \,  x_i\in M \right\}.
\end{equation}
By the above discussion it is possible to prove that the very low sublevels of $I_\rho$ have at least the homology of $M_k$, which is non-trivial. This in turn leads to a solution of \eqref{eq:liouv} for $\rho\notin 8\pi\N$.

\medskip

Let us pass now to the two-parameters case \eqref{eq}. The associated functional is defined by $J_\rho : H^1(M)\to\R$, $\rho=(\rho_1,\rho_2)$
\begin{equation} \label{func}
	J_\rho(u) = \frac{1}{2}\int_M |\nabla u|^2 \,dV_g - \rho_1  \left( \log\int_M h_1  e^u \,dV_g  - \int_M u \,dV_g \right) -  \rho_2 \left( \log\int_M h_2  e^{-a u} \,dV_g  + \int_M a u \,dV_g \right).
\end{equation}
In this framework there is a generalized Moser-Trudinger inequality obtained in \cite{ors} which can be rephrased as
\begin{equation} \label{m-t}
8\pi \log\int_M e^{u-\ov u} \, dV_g + \frac{8\pi}{a^2} \log\int_M e^{-a(u-\ov u)} \, dV_g \leq \frac 12 \int_M |\n u|^2\,dV_g + C_{M,g},\qquad \ov u= \fint_M u\,dV_g.
\end{equation}
We point out that we can interpret the latter sharp inequality by means of the minimum local blow-up mass obtained in the Theorem \ref{th:mass}. Concerning the existence issue to the general problem \eqref{eq} there are still a lot of gaps. If we restrict our attention to the symmetric case, namely the sinh-Gordon equation \eqref{eq:sinh}, there are some successful strategies that one could try to pursue also for the general equation. To this end, let us briefly illustrate them.

\medskip

In case one of the two parameters $\rho_i$ is small then we can rely on the analysis developed for the standard mean field equation \eqref{eq:liouv} (see the argument above) and get a solution to \eqref{eq:sinh} \cite{zhou}. When both parameters are large the situation is much more subtler due to the interaction of the two components $u$ and $-u$. In this direction an existence result is derived in \cite{jev} via a detailed description of the sublevels of the associated energy functional. Finally, a general existence result under the assumption the surface has positive genus is given in \cite{bjmr}, while the sphere case is still an open problem. By similar arguments as before, one can use improved Moser-Trudinger inequalities to show that if $\rho_1<8k\pi$, $\rho_2<8l\pi$, $k,l\in\N$, in the very low sublevels of the energy functional either $e^u$ is close to $M_k$ or $e^{-u}$ is close to $M_l$ in the distributional sense, recall \eqref{M_k}. This alternative can be expressed by using the \emph{topological join} of $M_k$ and $M_l$. The topological join of two topological sets is given by
\begin{equation}\label{join}
 A*B = \Bigr\{ (a,b,s): \; a \in A,\; b \in B,\; s \in [0,1]  \Bigr\} \Bigr/_{\hspace{-0.1cm}\mbox{\large {\emph{E}}}}\;,
\end{equation}
where $E$ is an equivalence relation such that:
$$
(a_1, b,1) \stackrel{E}{\sim} (a_2,b, 1)  \quad \forall a_1, a_2
\in A, b \in B
  \qquad \quad \hbox{and} \qquad \quad
(a, b_1,0)  \stackrel{E}{\sim} (a, b_2,0) \quad \forall a \in A,
b_1, b_2 \in B.
$$
Hence, the low sublevels of the functional are mapped into $M_k * M_l$: the join parameter $s$ somehow measures whether $e^{u}$ is closer to $M_k$ or $e^{-u}$ is closer to $M_l$. The assumption on $M$ to have positive genus is then used in a crucial way to construct two disjoint simple non-contractible curves  $\gamma_1, \gamma_2$ such that $M$ retracts on each of them through continuous maps $R_1, R_2$, respectively. By means of these retractions one can restrict the target from $M_k * M_l$ to $(\gamma_1)_k *(\gamma_2)_l$ only. The non-trivial homology of $(\gamma_1)_k *(\gamma_2)_l$ is then used to produce a solution to \eqref{eq:sinh}.

Finally, we point out that the general case \eqref{eq} with $a\in(-1,1)$ was treated in \cite{rtzz}: for some suitably small parameters they are able to derive existence of solutions to \eqref{eq} in a slightly supercritical regime. 

\medskip

The aim of the paper is to give general existence results in a supercritical case: more precisely, we will show existence of solutions to \eqref{eq} in the two supercritical regimes highlighted in Theorem \ref{th:comp}. The argument is based on two types of improved Moser-Trudinger inequalities and it works for any choice of the underlying surface (with the exception of the last result). The first result is the following one.
\begin{theorem} \label{th:ex1}
Suppose either $\rho_1<8\pi, \,\rho_2 \in \left(\dfrac{8\pi}{a^2}k, \dfrac{8\pi}{a^2}(k+1) \right)$ or $\rho_1 \in (8\pi k,8\pi(k+1)), \,\rho_2<\dfrac{8\pi}{a^2}$, for some $k\in\N$, where $a\in(0,1)$. Then, there exists a solution to \eqref{eq} for any underlying surface $M$.
\end{theorem}
The latter result follows mainly by the analysis developed for the one-parameter case \eqref{eq:liouv}, see for example \cite{zhou}, and it is based on a \emph{macroscopic} improved Moser-Trudinger inequality. 

\medskip

On the other hand, the second existence result concerns a doubly supercritical case, namely when both $\rho_1>8\pi$, $\rho_2 > \dfrac{8\pi}{a^2}$ and therefore it is more delicate to handle due to the non-trivial interaction between the two components $u$ and $-a u$. We have the following result.
\begin{theorem} \label{th:ex2}
Suppose $\rho_1 \in (8\pi, 16\pi)$ and either $\rho_2 \in \left(\dfrac{8\pi}{a^2}, \dfrac{16\pi}{a^2} \right)$ if $a\geq \dfrac 12$ or $\rho_2 \in \left(\dfrac{8\pi}{a^2}, \dfrac{8\pi}{a^2}+\dfrac{16\pi}{a} \right)$ if $a< \dfrac 12$. Then, there exists a solution to \eqref{eq} for any underlying surface $M$.
\end{theorem}

\begin{rem}
We need to take into account the different ranges $a\geq \dfrac 12$ and $a<\dfrac 12$ because of the the different blow-up local masses in the point (2) of Theorem \ref{th:mass}, see Section \ref{sec:exist2} for more details. 
\end{rem}

The argument is based on the description of the low sublevels of the functional $J_\rho$: the aim is to detect a change of topology between two sublevels. To this end we will see that one has to take into account not only
the location of the concentration points, but also the \emph{scale} of concentration, in the spirit of \cite{jev} (first used in \cite{mal-ru} for the Toda system, see also \cite{jkm}). Indeed, we will show that if $J_\rho(u)$ is large negative, then $e^u$ and $e^{-au}$ are either concentrated at different points or they are concentrated at the same point but with different scales of concentration. The argument is based on a new improved Moser-Trudinger inequality which, differently from before, it is scale-invariant. This gives some constraints on the maps from low-energy levels  into the topological join of the barycentric sets. We anticipate that we will have to consider the set (observing $M_1\cong M$)
$$
	M*M \setminus \left\{ \left(x,x,s=\frac 12\right)\,:\, x\in M \right\},
$$
where the set we are excluding is made of configurations with the same point and the same scale of concentration. The join parameter $s=\frac 12$ roughly expresses the two components are concentrating with the same scale.

\medskip

The only case left is $\rho_1 \in (8\pi, 16\pi)$ and $\rho_2 \in \left(\dfrac{8\pi}{a^2}+\dfrac{16\pi}{a}, \dfrac{16\pi}{a^2} \right)$ for $a< \dfrac 12$; since we do not expect an improved inequality as in the above argument to hold, we restrict our attention to positive genus surfaces and apply the strategy in \cite{bjmr}, see the idea below \eqref{join}.
\begin{theorem} \label{th:ex3}
Suppose $\rho_1 \in (8\pi, 16\pi)$ and $\rho_2 \in \left(\dfrac{8\pi}{a^2}+\dfrac{16\pi}{a}, \dfrac{16\pi}{a^2} \right)$ for $a< \dfrac 12$ and suppose $M$ has positive genus $g(M)>0$. Then, there exists a solution to \eqref{eq}.
\end{theorem}
The organization of this paper is as follows. In Section \ref{sec:blow-up} we analyze the blow-up limits and prove Theorem \ref{th:mass}. In Section \ref{sec:exist1} we show the strategy to get the existence results of the Theorem \ref{th:ex1} and Section \ref{sec:exist2} we introduce the argument which yields to the proof of the Theorem \ref{th:ex2}. In Section \ref{sec:genus} we prove the Theorem \ref{th:ex3}.

\

\

\begin{center}
\textbf{Notation}
\end{center}

\medskip

The symbol $B_r(p)$ will denote the open metric ball of
radius $r$ and center $p$. We will simply write $B_r\subset\R^2$ for balls which are centered at $0$, while $A_p(r_1, r_2)$ is the open
annulus of radii $r_1, r_2$ and center $p$. 

The average of $u\in H^1(M)$ is denoted by $\ov u = \fint_M u \,dV_g.$ For the sublevels of the functional $J_\rho$ we will write
\begin{equation} \label{sub}
	J_\rho^L = \bigr\{ u\in H^1(M) \,:\, J_\rho(u) \leq L  \bigr\}.
\end{equation}
Let $\mathcal M(M)$ be the  set of all Radon measures on $M$: we will consider the Kantorovich-Rubinstein distance 
\begin{equation} \label{dist}
	\dkr (\mu_1,\mu_2) = \sup_{\|f \|_{Lip}\leq 1} \left| \int_M f \, d\mu_1 - \int_M f\,d\mu_2 \right|, \qquad \mu_1,\mu_2 \in \mathcal M(M).
\end{equation}

Throughout the paper the letter $C$ will stand for positive constants which
are allowed to vary among different formulas and even within the same lines.
To stress the dependence of the constants on some
parameter  we add subscripts to $C$, for example $C_\d$. We will write $o_{\alpha}(1)$ to denote
quantities that tend to $0$ as $\alpha \to 0$ or $\alpha \to
+\infty$; the symbol
$O_\alpha(1)$ will be used for bounded quantities.

\

\section{Blow-up limits} \label{sec:blow-up}

We are concerned here with the study of blow-up limits to \eqref{eq} and with the proof of Theorem \ref{th:mass}. We will actually consider the following localized problem:
\begin{equation}
\label{seq}
	-\D u_k =  \rho_{1,k} h_1e^{u_{1,k}} - a  \rho_{2,k} h_2 e^{u_{2,k}} \qquad \mbox{in } B_1,
\end{equation}
with $(\rho_{1,k},\rho_{2,k})\to(\bar \rho_1,\bar \rho_2)$, where 
\begin{align} \label{norm}
u_{1,k} = u_k - \log \int_M h_{1} \,e^{u_{k}} \,dV_g, \qquad u_{2,k} = -a u_k - \log \int_M h_{2} \,e^{-a u_{k}} \,dV_g,
\end{align}
are such that $\int_M u_k \,dV_g = 0$ and $0$ is the only blow-up point in $B_1$, i.e.:
\begin{equation}
\label{a1}
\max_{K\subset\subset B_1\setminus\{0\}} u_{i,k}\leq C_K,\quad \max_{x\in B_1,\, i=1,2}\{u_{i,k}(x)\}\rightarrow\infty.
\end{equation}
To set the problem, we suppose that
\begin{equation}
\label{a2}
h_1(0)=h_2(0)=1, ~ \frac1C\leq h_i(x)\leq C, ~ \|h_i(x)\|_{C^3(B_1)}\leq C, \qquad \forall x\in B_1, ~i=1,2,
\end{equation}
for some constant $C>0$. Moreover, it is natural to assume bounded boundary oscillations
\begin{align}
\label{a3}
\begin{split}
|u_{i,k}(x)-u_{i,k}(y)|\leq C,\qquad \forall~x,y\in\partial B_1,
\end{split}
\end{align}
where $C$ is independent of $k.$ By the normalization in \eqref{norm} we may assume
\begin{equation} \label{bound}
\lim_{k\to+\infty}\frac{1}{2\pi}\int_{B_1} \rho_{i,k}h_i e^{u_{i,k}}\leq \frac{\bar\rho_i}{2\pi}.
\end{equation}	
Letting
\begin{align}
\label{localmass}
\sigma_i=\lim_{\delta\rightarrow0}\lim_{k\rightarrow\infty}\frac{1}{2\pi}\int_{B_{\delta}}\rho_{i,k} h_i e^{u_{i,k}},
\end{align}
the Theorem \ref{th:mass} is equivalent to proving that 
\begin{enumerate}
	\item If $\bar \rho_2 < \dfrac{8\pi}{a^2}$ (resp. $\bar \rho_1 < 8\pi$), then $(\s, \s_2)$ is given by
	\begin{equation} \label{mass1}
		(4, 0) \qquad \left(\mbox{resp. } \left(0,\dfrac{4}{a^2}\right)  \right).
	\end{equation}
	
	\item If $\bar\rho_1 < 16\pi$, $\bar \rho_2 < \dfrac{16\pi}{a^2}$, then $(\s_1, \s_2)$ is one of the following types:
	\begin{equation} \label{mass2}
	\begin{split}
	\begin{array}{ll}
	 \mbox{if } a\geq \dfrac 12 \,: &	\qquad (4, 0), \qquad \left(0,\dfrac{4}{a^2}\right), \vspace{0.1cm}\\
	 \mbox{if } a< \dfrac 12 \,:&	\qquad (4, 0), \qquad \left(0,\dfrac{4}{a^2}\right), \qquad \left(4,\dfrac{4}{a^2}+\dfrac{8}{a}\right).
	\end{array}
	\end{split}
	\end{equation} 
\end{enumerate}
One can see the details of the above localized argument for example in \cite{lin-zh}.

\medskip

We introduce now some preliminary tools: we refer to \cite{jwy} for the details. The starting point is the following process which select a finite number of bubbling disks where the blowing-up limits resemble globally define Liouville-type equations. One has just to point out that due to the opposite-sign structure in \eqref{seq} the argument can be carried out with minor modifications.
\begin{proposition}
\label{pr2.1}
Let $u_k$ be a sequence of blow-up solutions of \eqref{seq} such that \eqref{a1}, \eqref{a2} and \eqref{a3} hold true. Then there exists finite sequence of points $\Sigma_k=\{x_1^k,\cdots,x_m^k\}$ (all $x_j^k\rightarrow0,~j=1,\cdots,m$) and positive scales $l_1^k,\cdots,l_m^k\rightarrow0$ such that, letting $M_{k,j}=\max_{i=1,2}\{u_{i,k}(x_j^k)\}$, we have
\begin{enumerate}
  \item $M_{k,j}=\max_{B_{l_j^k}(x_j^k),\,i=1,2}\{u_{i,k}\}$~for~$j=1,\cdots,m.$
  \item $\exp\left(\frac12 M_{k,j}\right)l_j^k\rightarrow\infty$ for $j=1,\cdots,m.$
  \item Let $\varepsilon_{k,j}=e^{-\frac{1}{2}M_{k,j}}$. Setting
  \begin{align}
  \label{2.1}
  \begin{split}
  &v_i^k(y)=u_{i,k}\bigr(\varepsilon_{k,j}y+x_j^k\bigr)+2\log\varepsilon_{k,j} 
  \end{split} \qquad \mbox{in } B_{l_j^k}(x_j^k),
  \end{align}
  we have the following alternative:
  \begin{itemize}
  \item[(a)] either $v_1^k\to v_1$ in $C_{loc}^2(\R^2)$ which satisfies the equation $\Delta v_1+\bar\rho_1 e^{v_1}=0$ and $v_2^k\to-\infty$ over all compact subsets of $\R^2$ and 
  $$
  	\frac{1}{2\pi}\int_{B_{l_j^k}(x_j^k)}\rho_{1,k} h_1 e^{u_{i,k}}>4,
  $$
  \item[(b)] or $v_2^k \to v_2$ in $C_{loc}^2(\R^2)$ which satisfies the equation $\Delta v_2+a^2 \bar\rho_2 e^{v_2}=0$ and $v_1^k\to-\infty$ over all compact subsets of $\R^2$ and
  $$
  	\frac{1}{2\pi}\int_{B_{l_j^k}(x_j^k)}\rho_{2,k} h_2 e^{u_{2,k}}>\frac{4}{a^2}.
  $$
  \end{itemize}
  \item There exists a constant $C>0$ independent of $k$ such that
  \begin{align*}
  \max_{i=1,2}\{u_{i,k}(x)\}+2\log\mathrm{dist}(x,\Sigma_k)\leq C, \qquad \forall x\in B_1.
  \end{align*}
\end{enumerate}
\end{proposition}

We point out that due to the local mass in the point (3) of the latter result and the bound \eqref{bound} the process stops after a finite number of steps. Moreover, by using the point (4) one can get a Harnack-type inequality outside the bubbling disks as follows.
\begin{proposition}
\label{pr2.2}
Letting $x_0\in B_1\setminus\Sigma_k,$ there exists $C>0$ independent of $x_0$ and $k$ such that
\begin{align*}
|u_{i,k}(x_1)-u_{i,k}(x_2)|\leq C\qquad\forall x_1,x_2\in B_{d(x_0,\Sigma_k)/2}(x_0), \quad i=1,2.
\end{align*}
\end{proposition}
The latter estimates gives us bounded oscillation away from the blow-up disks and hence the behavior of a solution can be encoded in its spherical average. More precisely, let $x_k\in\Sigma_k$ and $\tau_k=\frac12d({x_k,\Sigma_k\setminus\{x_k\}})$, then for $x,y\in B_{\tau_k}(x_k)$ and $|x-x_k|=|y-x_k|$ we have $u_{i,k}(x)=u_{i,k}(y)+O(1)$ and hence $u_{i,k}(x)=\overline{u}_{i,x_k}(r)+O(1)$ where $r=|x_k-x|$ and
$$\overline{u}_{i,x_k}(r)=\frac{1}{2\pi}\int_{\partial B_r(x_k)}u_{i,k}.$$
We will see in the sequel how to use this property.

\medskip

In each bubbling disk of Proposition \ref{pr2.1} one can  derive some information on the local mass by means of a Pohozaev-type identity. We can easily show that 
\begin{align}
\label{2.2}
\begin{split}
&\int_{B_r}\left(\rho_{1,k} x\cdot\nabla h_1e^{u_{1,k}}+\rho_{2,k}x\cdot\nabla h_2e^{u_{2,k}}\right)+2\int_{B_r}\left(\rho_{1,k}h_1e^{u_{1,k}}+\rho_{2,k}h_2e^{u_{2,k}}\right)\\
&=r\int_{\partial B_r}\left(|\partial_{\nu}u_{1,k}|^2-\frac12|\nabla u_{1,k}|^2\right)+r\int_{\partial B_r}\left(\rho_{1,k}h_1e^{u_k}+\rho_{2,k}h_2e^{u_{2,k}}\right).
\end{split}
\end{align}
Consider now the above identity in $B_{l_j^k}(x_j^k)$ and let
\begin{align*}
\tilde{\sigma}_i^k(l_j^k) = \frac{1}{2\pi}\int_{B_{l_j^k}(x_j^k)}\rho_{i,k}h_i e^{u_{i,k}}
\end{align*}
be the local masses in this ball. To estimate the second term on the right hand side of \eqref{2.2} it is useful to give the following definition: we say $u_{i,k}$  has fast decay at $x\in B_1$ if
$$u_{i,k}(x)+2\log\mathrm{dist}(x,\Sigma_k)\leq -N_k, $$
hold for some $N_k\rightarrow+\infty$. If instead
$$u_{i,k}(x)+2\log\mathrm{dist}(x,\Sigma_k)\geq -C, $$
for some $C>0$ independent of $k$, we say $u_{i,k}$ has a slow decay at $x$. If both $u_{i,k}$, $i=1,2$ have fast decay on $\partial B_{l_j^k}(x_j^k)$  the second term on the right hand side of \eqref{2.2} is $o(1)$. It is indeed possible to show that from \eqref{2.2} we get the identity
\begin{equation} \label{poh}
4\left(\tilde{\sigma}_1^k(l_j^k)+\tilde{\sigma}_2^k(l_j^k)\right)=\bigr(\tilde{\sigma}_1^k(l_j^k)-a\tilde{\sigma}_2^k(l_j^k)\bigr)^2+o(1),
\end{equation}
see \cite{jwy} for the details.

\medskip

We are now in the position to prove the local mass Theorem \ref{th:mass}.

\begin{proof}[Proof of Theorem \ref{th:mass}.]
We need to derive the values in \eqref{mass1}, \eqref{mass2}. We follow here the argument in \cite{jwy} with some modification, so we will be sketchy. Let $x_j^k\in\Sigma_k$, where $\Sigma_k$ is obtained in Proposition \ref{pr2.1}, and suppose for simplicity that $x_j^k=0$. Let $\tau_k=\frac12\mathrm{dist}(0,\Sigma_k\setminus\{0\})$, set
$$\sigma_i^k(r,x_j^k)=\sigma_i^k(r)=\frac{1}{2\pi}\int_{B_r(0)}\rho_{i,k}h_i e^{u_{i,k}},$$
for $0<r\leq\tau_k$ and $\overline{u}_{i,k}(r)=\frac{1}{2\pi r}\int_{\partial B_r(0)}u_{i,k}.$ A useful observation is the following:
\begin{equation} \label{deriv}
\begin{split}
&\frac{d}{dr}\overline{u}_{1,k}(r)=\frac{1}{2\pi r} \int_{\p B_r} \frac{\p u_{1,k}}{\p \nu} = \frac{1}{2\pi r} \int_{B_r} \D u_{1,k} = \frac{-\sigma_1^k(r)+a\sigma_2^k(r)}{r}\,, \\
&\frac{d}{dr}\overline{u}_{2,k}(r) = a \frac{\sigma_1^k(r)-a\sigma_2^k(r)}{r}\,.
\end{split}
\end{equation}

\medskip

\no \textbf{Proof of (1).} We prove here \eqref{mass1}. We are assuming $\bar\rho_2 < \dfrac{8\pi}{a^2}$ (for the other alternative we can reason in the same way). By Proposition \ref{pr2.1} we observe that
\begin{align*}
\max_{i=1,2}\{u_{i,k}(x)\}+2\log|x|\leq C, \qquad |x|\leq \tau_k,
\end{align*}
and letting $-2\log\delta_k=\max_{x\in B_{\tau_k}(0)}\max_{i=1,2}\{u_{i,k}(x)\}$ we consider
\begin{align*}
\begin{split}
v_i^k(y)=u_{i,k}(\delta_ky)+2\log \delta_k, 
\end{split} \qquad |y|\leq\tau_k/\delta_k.
\end{align*}
As in Proposition \ref{pr2.1} one of $v_i^k$ converges and the other one tends to minus infinity over the compact subsets of $\mathbb{R}^2$. Suppose that $v_1^k\to v_1$ in $C_{loc}^2(\mathbb{R}^2)$ and $v_2^k\to-\infty$ over any compact subset of $\R^2$, where $v_1$ satisfies $\D v_1 + \bar\rho_1 e^{v_1}=0$ in $\R^2$. Then by the classifications result of the latter equation one can choose $R_k\rightarrow\infty$ such that
\begin{equation}
\label{2.5}
\sigma_1^k(\delta_kR_k)=\frac{1}{2\pi}\int_{B_{R_k}}\rho_{1,k}	h_1(\delta_ky)\,e^{v_1^k}=4+o(1), \qquad \sigma_1^k(\delta_kR_k)=\frac{1}{2\pi}\int_{B_{R_k}}\rho_{2,k}	h_2(\delta_ky)\,e^{v_2^k}=o(1).
\end{equation}
Then we get $\sigma_1^k(\delta_kR_k)=4+o(1)$ and $\sigma_2^k(\delta_kR_k)=o(1).$ If instead $v_2^k\to v_2$ with $\D v_2 + \bar\rho_2 e^{v_2}=0$ we would get $\sigma_2^k(\delta_kR_k)=\dfrac{4}{a^2}+o(1)$ and $\sigma_1^k(\delta_kR_k)=o(1).$ The latter estimate is not possible by the assumption on $\bar\rho_2$ and the bound \eqref{bound}.

Now, we need to consider the energy's increasing from $B_{\delta_kR_k}$ to $B_{\tau_k}$. Observe that on $\partial B_{\delta_kR_k}$ by \eqref{deriv} and \eqref{2.5} we have
$$\frac{d}{dr}(\overline{u}_{2,k}(r)+2\log r)>0,$$
in other words $u_{2,k}$ may become a slow decay component when $r$ increases. If this does not happen the energy does not change and we keep having
\begin{equation}
\label{2.6}
\sigma_1^k(\tau_k)=4+o(1), \qquad \sigma_2^k(\tau_k)=o(1),
\end{equation}
see for example \cite{jwy}.

\medskip

If $u_{2,k}$ becomes slow decaying before $r$ reaches $\tau_k$, i.e
\begin{equation*}
\bar u_{2,k}(s)+2\log s\geq -C,
\end{equation*}
for some $C>0$, then $u_{2,k}$ starts to increase its energy. In the case $\tau_k \cong s$ by Proposition \ref{pr2.2} we still get $\sigma_1^k(\tau_k)=4+o(1)$ with $u_{1,k}$ fast decay, while $u_{2,k}$ has slow decay and its local mass can not be evaluated at this point. If instead $\tau_k\gg s$  we can find $N>1$ such that on $\partial B_{Ns}$
\begin{align}
&\sigma_2^k(Ns)\geq \dfrac{2}{a^2} + \dfrac 4a, \qquad \sigma_1^k(Ns)=4+o(1),\label{est}\\
&\frac{d}{dr}(\bar u_{2,k}(r)+2\log r)\mid_{r=Ns}<0, \qquad \frac{d}{dr}(\bar u_{1,k}(r)+2\log r)\mid_{r=Ns}>0,\nonumber
\end{align}
see \cite{jwy}. The idea is that from $r=s$ to $r=Ns$ the energy of $u_{2,k}$ increases and hence the derivative the associated derivative changes from positive to negative by \eqref{deriv}. On the other hand, By Proposition~\ref{pr2.2} $u_{1,k}$ still has fast decay and hence its energy does not change. At this point it is possible to take $N_k$ tending to $+\infty$ slowly such that and on $\p B_{N_k s}$ both $u_{i,k}$, $i=1,2$  have fast decay. Therefore, we can use the Pohozaev identity \eqref{poh} in $B_{N_k s}$ and \eqref{est} to obtain
\begin{equation} \label{est2}
\sigma_{1}^k(N_k s)=4+o(1), \qquad \sigma_2^k(N_k s)=\frac{4}{a^2}+\frac{8}{a} + o(1).
\end{equation}
Again, the latter estimate is not possible by the assumption on $\bar\rho_2$ and the bound \eqref{bound}.

\medskip

It follows that in each bubbling disk we have $\sigma_1^k(\tau_k)=4+o(1)$ with $u_{1,k}$ fast decay. Since it has fast decay, by Proposition \ref{pr2.2} the energy contribution of $u_{1,k}$ comes just from the bubbling disks, more precisely
\begin{equation} \label{fast}
	\sigma_{1}^k(r) = \sum_{j=1}^k \sigma_{1}^k(\t_k^j,x_k^j) + o(1) = 4k + o(1).
\end{equation}
As before, by taking $r$ suitable we have both $u_{i,k}$, $i=1,2$  have fast decay and we can use the Pohozaev identity \eqref{poh} in $ B_{r}$. By the latter estimate, by the assumption on $\bar\rho_2$ and by the bound \eqref{bound} the only possibility is 
$$
	(\sigma_{1}^k(r), \sigma_{2}^k(r)) = (4,0) + o(1),
$$
which concludes the proof of (1) of Theorem \ref{th:mass}.

\

\no \textbf{Proof of (2).} We prove here \eqref{mass2}. We assume now $\bar\rho_1 < 16\pi$, $\bar \rho_2 < \dfrac{16\pi}{a^2}$ and we proceed as in the proof of (1). In the first step we observed that one of $v_i^k$ converges and the other one tends to minus infinity over the compact subsets of $\mathbb{R}^2$. In this case both situations are possible. Then, analogously as in \eqref{2.5} we have $\bigr(\sigma_1^k(\delta_kR_k), \sigma_2^k(\delta_kR_k)\bigr)$ is a $o(1)$ perturbation of
\begin{equation} \label{est4}
	(4,0) \qquad \mbox{or} \qquad \left(0,\frac{4}{a^2}\right).
\end{equation}
Reasoning as in the proof of (1) concerning \eqref{est2}, starting from $(4,0)$ we could get
\begin{equation} \label{est3}
 \left( 4, \frac{4}{a^2}+\frac{8}{a} \right),
\end{equation} 
while starting from $\left(0,\dfrac{4}{a^2}\right)$ we would get $\left(\dfrac{8}{a}+{4},\dfrac{4}{a^2}\right)$. The latter possibility can not happen due to the assumption on $\bar\rho_1$ and by the bound \eqref{bound}. Hence, let us focus on \eqref{est3}. We first point out that the latter local mass is present just for $a<\dfrac 12$ due to the assumption on $\bar\rho_2$ and by the bound \eqref{bound}. At this point the role of $u_{1,k}$ and $u_{2,k}$ is exchanged in the above argument: this means that if there is a change in the local mass then the mass of $u_{2,k}$ changes by $o(1)$ while the one of $u_{1,k}$ jumps according to the Pohozaev identity \eqref{poh}. However, it is easy to see that by the assumption on $\bar\rho_1$ and by the bound \eqref{bound} such a jump is not possible. Therefore, we exhausted all the possibilities with \eqref{est4} and \eqref{est3}. Moreover, as before the component with bigger local mass has fast decay property.

\medskip

Finally, we have to combine the bubbling disks. We claim there is at most one bubbling disk. Indeed, using both the assumptions on $\bar\rho_i$, $i=1,2$ and the bound \eqref{bound} we can not have two bubbling disks with the same type of local mass. Moreover, by the same reason we can not have both \eqref{est4} and \eqref{est3} type disks. The only possibility to have two disks is the case with $(4,0)$ and $\left(0,\dfrac{4}{a^2}\right)$ (actually, to be more precise we would just know in one disk $\s_1^k(r_1,x_i^k)=4+o(1)$ with $u_{1,k}$ fast decaying and in the other disk $\s_2^k(r_2,x_j^k)=\dfrac{4}{a^2}+o(1)$ with $u_{2,k}$ fast decaying: however, we can treat this case in a similar way). Since in the first group $u_{1,k}$ has fast decay and in the second one $u_{2,k}$ has fast decay we deduce by Proposition~\ref{pr2.2} that both $u_{i,k}$ have fast decay. Similarly as in \eqref{fast} one can derive that the combination of this two disks yields to a local mass of the type
$$
	\left(4,\dfrac{4}{a^2}\right).
$$ 
On the other hand, since both $u_{i,k}$ have fast decay we can apply the Pohozaev identity \eqref{poh} to rule out the latter possibility. 

It follows that we have just one bubbling disk and the possibly local masses are those in \eqref{est4} and \eqref{est3}. 

\end{proof}

\medskip

\section{Existence results in the semi-coercive case} \label{sec:exist1}

Aim of this section is to present the variational argument and to derive the first existence result, see Theorem \ref{th:ex1}. The strategy relies on the ideas developed for the standard mean field equation \eqref{eq:liouv}: in \cite{zhou} it was used for the sinh-Gordon case \eqref{eq:sinh}. Therefore, we will presents just the main steps, highlighting the differences. We will focus on the case $\rho_1<8\pi, \,\rho_2 \in \left(\dfrac{8\pi}{a^2}k, \dfrac{8\pi}{a^2}(k+1) \right)$ since it requires a few new estimates. Some of the tools will be presented with full details since they will be used in the next section. 

\medskip

We start by stating two important corollaries of the compactness result in Theorem \ref{th:comp}. First, we observe there exists a high sublevel $J_\rho^L$ (recall the notation in \eqref{sub}) containing all the critical points of the functional. Then, by deforming the space $H^1(M)$ onto such sublevel we obtain the following result.
\begin{pro} \label{high}
Let $K$ be as in Theorem \ref{th:comp}. Suppose $\rho=(\rho_1,\rho_2)\in K$. Then, for some large $L>0$, $J_\rho^L$ is a deformation retract of $H^1(M)$. In particular it is contractible.
\end{pro}

\medskip

The Palais-Smale condition can by bypassed by using the compactness result as in \cite{lucia} to get the following property.
\begin{pro} \label{top-arg}
Let $K$ be as in Theorem \ref{th:comp}. Let $a,b\in\R$ be such that $a<b$ and $J_\rho$ has no critical points $u\in H^1(M)$ with $a\leq J_\rho(u) \leq b$. Suppose $\rho=(\rho_1,\rho_2)\in K$. Then, $J_\rho^a$ is a deformation retract of $J_\rho^b$.
\end{pro}

\medskip

Our aim will be then to show that the very low sublevels of $J_\rho$ have non-trivial homology. The first tool we need in this direction is an improved version of the Moser-Trudinger inequality \eqref{m-t}. To this end we state a local version of it, in the spirit of \cite{jev,mal-ru}.
\begin{pro}\label{mt-local}
Let $\delta>0$ and $\Omega_1\subset\Omega_2\subset M$ be such that
$d(\Omega_1, \partial \Omega_2) \geq \delta$. Then, for any $\e > 0$
there exists a constant $C = C(\e, \delta)$ such that for all $u
\in H^1(M)$
$$ 
8\pi \log \int_{\Omega_1}  e^{u-\fint_{\O_2} u}\, dV_g + \frac{8\pi}{a^2}\log \int_{\Omega_1} e^{-a\left(u-\fint_{\O_2} u\right)}\, dV_g \leq \frac{1}{2} \int_{\Omega_2} |\nabla u|^2 \,dV_g + \e\int_{M} |\nabla u|^2 \,dV_g + C. 
$$
\end{pro}  

\begin{proof}
We may assume $\fint_{\O_2} u=0$ and we decompose $u$ so that $u=v+w$, with $\fint_{\O_2} v = \fint_{\O_2} w = 0$ and $v\in L^{\infty}(\O_2)$. Such decomposition will be suitably chosen later on. Let $0\leq\chi \leq 1$ be a cut-off function such that
$$
	{\chi}_{|\O_1} \equiv 1, \quad {\chi}_{|M\setminus (B_{\d/2}(\O_1))}\equiv 0, \quad |\n \chi| \leq C_{\d}.
$$
Then we have
$$
\log \int_{\O_1} e^u \,dV_g  \leq \log \int_{M} e^{\chi w} \,dV_g + \|v\|_{L^{\infty}(\O_1)}, 
$$
and similarly for $-a u$. Using the latter estimate and the Moser-Trudinger inequality \eqref{m-t} for $\chi w$ we obtain
\begin{align} \label{mt1}
\begin{split}
	8\pi \log \int_{\O_1} e^u\,dV_g + \frac{8\pi}{a^2} \log \int_{\O_1} e^{-au}\,dV_g  \leq &~ \frac 12 \int_M |\n(\chi w)|^2 \,dV_g+ C\|v\|_{L^{\infty}(\O_1)}\\
	  & + 8\pi \int_M \chi w \,dV_g  - \frac{8\pi}{a^2}\int_M a\chi w \,dV_g + C.
\end{split}
\end{align}
By the Poincar\'e's and Young's inequalities we get
\begin{equation} \label{mt2}
 \int_M \chi w \,dV_g  \leq \e \int_{\O_2} |\n w|^2 \,dV_g + \e \int_{\O_2}  w^2 \,dV_g+C_{\e}.
\end{equation}
Moreover, by Young's inequality it holds
\begin{align}
	\int_M |\n(\chi w)|^2 \,dV_g \leq (1+\e)\int_{\O_2}  |\n w|^2 \,dV_g + C_{\e,\ov\d}\int_{\O_2} w^2 \,dV_g. \label{mt3}
\end{align}
By combining \eqref{mt1}, \eqref{mt2}, \eqref{mt3}  we end up with
\begin{align}
8\pi \log \int_{\O_1} e^u\,dV_g + \frac{8\pi}{a^2} \log \int_{\O_1} e^{-au}\,dV_g  \leq  \frac{1+\e}{2}\int_{\O_2}  |\n w|^2 \,dV_g + C\|v\|_{L^{\infty}(\O_1)} + C\int_{\O_2} w^2 \,dV_g 
\label{mt-f3}
\end{align}
Reasoning as in Proposition 2.3 in \cite{mal-ru} we can then choose $v,w$ by decomposing $u$ with respect to a basis of eigenfunctions of $-\D$ in $H^1(\O_2)$ with Neumann boundary conditions to estimate the left terms in \eqref{mt-f3}. 
\end{proof}

\medskip

By the latter result we derive an improved inequality whenever the functions $e^u, e^{-au}$ are spread over the surface: indeed, it is sufficient to apply the localized inequality of Proposition \ref{mt-local} around each region which contains a portion of the \emph{volume} of $e^u, e^{-au}$, see \cite{bjmr, je-ya}.
\begin{pro}\label{mt-impr}
Let $\dis{\d>0}$, $\dis{\th>0}$, $\dis{k,l\in\N}$ and
$\dis{\{\O_{1,i},\O
_{2,j}\}_{i\in\{1,\dots,l\},j\in\{1,\dots,k\}}\subset M}$ be such that
$$
d(\O_{1,i},\O_{1,i'})\ge\d, \quad d(\O_{2,j},\O_{2,j'}) \ge\d,\qquad\forall\;i, \ i'\in\{1,\dots,l\}\hbox{ with }i\ne i' \;, \forall\;j,\ j' \in\{1,\dots,k\}\hbox{ with }j\ne j'.
$$ 
Then, for any $\dis{\e>0}$ there exists $\dis{C=C\left(\e,\d,\th,k,l,M\right)}$ such that if $u\in H^1(M)$ satisfies
\begin{align*}
\int_{\O_{1,i}} e^{u}\,dV_g\ge\th\int_M e^{u}\,dV_g,~\forall i\in\{1,\dots,l\},
\qquad \int_{\O_{2,j}} e^{-au}\,dV_g\ge\th\int_M e^{-au}\,dV_g,~\forall j\in\{1,\dots,k\},
\end{align*}
it follows that
$$8\pi l\log\int_M e^{u-\ov{u}}\,dV_g+\frac{8\pi}{a^2}k\log\int_M e^{-a(u-\ov u)}\,dV_g\leq \frac{1+\e}{2}\int_M |\n u|^2\,dV_g+C. $$
\end{pro} 

\medskip

It follows that the more the functions $e^u, e^{-au}$ are spread the better bounds we have on $J_\rho$. On the other way round, if we are very low in the energy then $e^u, e^{-au}$ should be concentrated around some points. More precisely, in the semi-coercive case at least one of the two components have to be concentrated as follows (see \cite{bjmr,zhou}). Recall the definition of $M_k$, $\dkr$ in \eqref{M_k}, \eqref{dist}, respectively.
\begin{pro}\label{p:map}
Suppose $\rho_1<8\pi, \,\rho_2 \in \left(\dfrac{8\pi}{a^2}k, \dfrac{8\pi}{a^2}(k+1) \right)$, $k\in\N$. Then, for any
$\displaystyle{\e>0}$ there exists $\dis{L>0}$ such that if
$\dis{u\in J_\rh^{-L}}$ then
$$\dkr\left(\frac{\
h_2e^{-au}}{\int_M
h_2e^{-au}\,dV_g},M_k\right)<\e.
$$
Moreover, for $L$ sufficiently large there exists a continuous retraction
$$
	\Psi: J_\rho^{-L} \to M_k.
$$
\end{pro}

\medskip

On the other hand, it is possible to construct a reverse map $\Phi: M_k \to J_\rho^{-L}$. Indeed, letting $\s := \sum_{i=1}^k t_i \d_{x_i} \in M_k$, for large  $\l>0$ we set
\begin{equation} \label{bubble}
    \var_{\l , \s} (x) =  - \frac 1a \log \, \sum_{i=1}^{k} t_i \left( \frac{1}{1 + \l^2 d(x,x_i)^2} \right)^2,
\end{equation}
and define $\Phi=\Phi_\l: M_k \to H^1(M)$ by $\Phi_\l(\s) = \var_{\l , \s}$. Then, we have the following property.
\begin{pro} \label{p:test}
Let $\var_{\l , \s}$ be defined in \eqref{bubble}. Suppose $\rho_1<8\pi, \,\rho_2 \in \left(\dfrac{8\pi}{a^2}k, \dfrac{8\pi}{a^2}(k+1) \right)$, $k\in\N$. Then, it holds
\begin{equation} \label{claim}
  J_{\rho}(\var_{\l , \s}) \to - \infty \quad \hbox{ as  } \l \to + \infty,
  \qquad \quad \hbox{ uniformly in }  \s \in M_k.
\end{equation}
Letting $\Phi$ be defined as above, it follows that for any $L>0$ there exists $\l>0$ large such that
$$
	\Phi: M_k \to J_\rho^{-L}.
$$
\end{pro}

\begin{proof}
We have the estimates
\begin{equation} \label{gr1}
    |\n \var_{\l , \s}(x)| \leq C \l, \qquad \mbox{for every $x\in M$},
\end{equation}
where $C$ is a constant independent of $\l$, $\s \in M_k$, and
\begin{equation} \label{gr2}
    |\n \var_{\l , \s}(x)| \leq \frac 1a \frac{4}{d_{min}(x)}, \qquad \mbox{for every $x\in M,$}
\end{equation}
where $\dis{d_{min}(x) = \min_{i=1,\dots,k} d(x,x_i)}$.

Indeed, it holds
$$
    \n \var_{\l , \s}(x) =  \frac 2a \l^2 \frac{\sum_{i=1}^k t_i \bigr(1 + \l^2 d^2(x,x_i)\bigr)^{-3} \n \bigr(d^2(x,x_i)\bigr)}{\sum_{j=1}^k t_j \bigr(1 + \l^2 d^2(x,x_j)\bigr)^{-2}}.
$$
By the fact that $\left|\n \bigr(d^2(x,x_i)\bigr)\right| \leq 2 d(x,x_i)$ and
$$
    \frac{\l^2 d(x,x_i)}{1 + \l^2 d^2(x,x_i)} \leq C \l, \qquad i = 1, \dots, k,
$$
with $C$ a fixed constant, we get (\ref{gr1}). Next, we have
\begin{eqnarray*}
  |\n \var_{\l , \s}(x)| & \leq & \frac 4a \l^2 \frac{\sum_{i=1}^k t_i \bigr(1 + \l^2 d^2(x,x_i)\bigr)^{-3} d(x,x_i)}{\sum_{j=1}^k t_j \bigr(1 + \l^2 d^2(x,x_j)\bigr)^{-2}} \leq \frac 4a \l^2 \frac{\sum_{i=1}^k t_i \bigr(1 + \l^2 d^2(x,x_i)\bigr)^{-2} \frac{d(x,x_i)}{\l^2 d^2(x,x_i)}}{\sum_{j=1}^k t_j \bigr(1 + \l^2 d^2(x,x_j)\bigr)^{-2}} \\
   & \leq & \frac 4a \frac{\sum_{i=1}^k t_i \bigr(1 + \l^2 d^2(x,x_i)\bigr)^{-2} \frac{1}{d_{\,1,min}(x)}}{\sum_{j=1}^k t_j \bigr(1 + \l^2 d^2(x,x_j)\bigr)^{-2}} = \frac 1a \frac{4}{d_{\,1,min}(x)},
\end{eqnarray*}
which gives (\ref{gr2}).

From (\ref{gr1}) we get
\begin{align*}
    \frac{1}{2} \int_{M} |\n \var_{\l , \s}(x)|^2 \,dV_g \leq \frac{1}{2}\int_{M \setminus \bigcup_i B_{\frac{1}{\l}}(x_i)} |\n \var_{\l , \s}(x)|^2 \,dV_g + C. 
\end{align*}
Letting
$$
    A_i = \left\{ x \in M : d(x,x_i) = \min_{j=1,\dots,k} d(x,x_j) \right\},
$$
by (\ref{gr2}) we derive
\begin{align*}
    \frac 12 \int_{M \setminus \bigcup_i B_{\frac{1}{\l}}(x_i)} |\n \var_{\l , \s}(x)|^2 \,dV_g
     \leq  \frac {8}{a^2} \sum_{i=1}^k \int_{A_i \setminus B_{\frac{1}{\l}}(x_i)} \frac{1}{d_{min}^2(x)} \,dV_g + C  \leq \left(\frac{16\pi}{a^2} k+o_\l(1)\right) \log \l + C.
\end{align*}
We end up with
\begin{equation}\label{grad}
\frac{1}{2} \int_{M} |\n \var_{\l , \s}(x)|^2 \,dV_g \leq \left(\frac{16\pi}{a^2}k+o_\l(1)\right) \log \l + C.
\end{equation}

We consider now the nonlinear term $e^{-a \var_{\l , \s}}$. To deduce the leading term it is enough to consider
$$
\int_M \frac{1}{\bigr( 1 + \l^2 d(x,\ov{x})^2 \bigr)^2} \,dV_g,
$$
for some fixed $\ov x\in M$. By a change of variables one readily has
$$
\int_M \frac{1}{\bigr( 1 + \l^2 d(x,\ov{x})^2 \bigr)^2} \,dV_g= \l^{-2}(1 + O(1)),
$$
which gives 
\begin{equation} \label{exp1}
\log \int_M  e^{-a\var_{\l , \s}} \,dV_g =  - 2 \log \l + O(1).
\end{equation}
We are left with $\int_M a\,\var_{\l , \s}\,dV_g$. For simplicity we suppose $k=1$. We have
$$
	{a \,\var_{\l , \s}}(x) = 4 \log \bigr( \max\{ 1,\l d(x, x_1) \} \bigr) + O(1), \qquad x_1 \in M.
$$
Therefore
\begin{align*}
	\int_M {a\,\var_{\l , \s}}\, dV_g &= 4 \int_{M \setminus B_{\frac{1}{\l}}(x_1)} \log \bigr(\l d(x, x_1)\bigr) \,dV_g + 4 \int_{B_{\frac{1}{\l}}(x_1)}\, dV_g + O(1) \\
	&= 4 \log\l \left|M \setminus B_{\frac{1}{\l}}(x_1)\right| +4 \int_{M \setminus B_{\frac{1}{\l}}(x_1)} \log (d(x, x_1)) \,dV_g + O(1).
\end{align*}
Recall that $|M|=1$. It follows
\begin{equation} \label{aver}
	\int_M {a\,\var_{\l , \s}} \, dV_g = \bigr(4+o_\l(1)\bigr) \log \l + O(1).
\end{equation}
Finally, using first the Jensen's inequality involving the part $e^{(\var_{\l , \s}-\ov \var_{\l , \s})}$ and then \eqref{grad}, \eqref{exp1}, \eqref{aver} we deduce
\begin{align*}
	J_\rho(\var_{\l , \s}) =~&\frac{1}{2}\int_M |\nabla \var_{\l , \s}|^2 \,dV_g - \rho_1  \left( \log\int_M h_1  e^{\var_{\l , \s}} \,dV_g  - \int_M \var_{\l , \s} \,dV_g \right)  \\
	&  -  {\rho_2}  \left( \log\int_M h_2  e^{-a \var_{\l , \s}} \,dV_g  + \int_M  a\, \var_{\l , \s} \,dV_g \right) \\
			\leq~&\frac{1}{2}\int_M |\nabla \var_{\l , \s}|^2 \,dV_g  -  {\rho_2}  \left( \log\int_M h_2  e^{-a \var_{\l , \s}} \,dV_g  + \int_M  a\, \var_{\l , \s} \,dV_g \right)\\
			\leq~&\left(\frac{16\pi}{a^2}k-2\rho_2+o_\l(1)\right) \log\l + O(1).
\end{align*}
By assumption $\rho_2>\dfrac{8\pi}{a^2}k$ and hence we get the desired property.
\end{proof}

\medskip

We prove now the main result of this section.

\begin{proof}[Proof of Theorem \ref{th:ex1}]
Let $\Psi:J_\rho^{-L}\to M_k$ and $\Phi:M_k\to J_\rho^{-L}$ be the maps defined in Proposition \ref{p:map} and before Proposition \ref{p:test}, respectively. The existence of solutions to \eqref{eq} will follow by showing that $\Psi\circ\Phi\cong Id_{M_k}$ (homotopic equivalence).

Indeed, it is well-known that
$$
\frac{h_2 e^{-a\var_{\l , \s}}}{\int_M h_2 e^{-a \var_{\l , \s}}\,dV_g} \rightharpoonup \s, \qquad \mbox{as } \l\to+\infty,
$$
in the sense of measures, see  \cite{bjmr,mal, zhou}. We have then just to observe that $\Psi$ is a retraction and hence $\Psi\left(\frac{e^{\var_{\l , \s}}}{\int_M e^{\var_{\l , \s}}\,dV_g}\right)$ will tend strongly to the configuration $\s$. The homotopy equivalence is obtained by letting $\l\to+\infty$.

Passing to the associated maps between the homology groups we derive
$$
	H_q(M_k) \hookrightarrow H_q \left(J_{\rho}^{-L}\right).
$$
Clearly, $M_k$ has non-trivial homology, see for example \cite{mal}, and in turn $J_{\rho}^{-L}$ has non-trivial homology as well. If by contradiction \eqref{eq} has no solutions we can apply Proposition \ref{top-arg} to get that $J_\rho^{-L}$ is a deformation retract of $J_\rho^{L}$ for any $L>0$ (recall that $\rho=(\rho_1,\rho_2)\in K$). Since $J_\rho^{L}$ is contractible for some $L$, see Proposition \ref{high}, $J_\rho^{-L}$ has trivial homology. We are lead to a contradiction.
\end{proof}

\medskip

\section{Existence result in a supercritical case} \label{sec:exist2}

In this section we present the strategy to deal with the doubly supercritical case and to prove the existence result of Theorem \ref{th:ex2}. As in the previous section our goal is to show the low sublevels of $J_\rho$ carry some non-trivial topology. To describe the configurations in the low sublevels we will first introduce the topological join, see \eqref{join}. Secondly, we need a new improved version of the Moser-Trudinger inequality \eqref{m-t} which is scaling invariant, differently from the one in Proposition \ref{mt-impr}, which will impose new constrains on the projection of the low sublevel onto the topological join, see the argument in the sequel.

We start by pointing out the role of the topological join. Let $\rho_1 \in (8\pi, 16\pi)$ and $\rho_2 \in \left(\dfrac{8\pi}{a^2}, \dfrac{16\pi}{a^2} \right)$. By the improved inequality in Proposition \ref{mt-impr} one can readily see that if $J_\rho(u)$ is large negative, then either $e^u$ or $e^{-au}$ (or both) need to be concentrated around a point of the surface: more precisely, they are $\dkr$-close to some $\d_p$ (recall the definition of $\dkr$ in \eqref{dist}). Recalling that $M_1\cong M$, it is then natural to map them into the join $M *M$. We collect these arguments in the following result, see \cite{bjmr, je-ya}. 
\begin{pro}\label{p:altern}
Suppose $\rho_1 \in (8\pi, 16\pi)$ and $\rho_2 \in \left(\dfrac{8\pi}{a^2}, \dfrac{16\pi}{a^2} \right)$ with $a\in(0,1)$. Then, for any
$\displaystyle{\e>0}$, there exists $\dis{L>0}$ such that any
$\dis{u\in J_\rh^{-L}}$ verifies either
\begin{equation} \label{al}
d_1=\dkr\left(\frac{\
h_1e^{u}}{\int_M
h_1e^{u}\,dV_g},M\right)<\e\qquad\qquad\text{or}\qquad\qquad
d_2=\dkr\left(\frac{ h_2e^{-au}}{\int_M
h_2e^{-au}\,dV_g},M\right)<\e.
\end{equation}
Moreover, for $L$ sufficiently large there exists a continuous map
$$
	\ov \Psi: J_\rho^{-L} \to M *M.
$$
\end{pro}

\begin{proof}
We sketch the main steps for the reader's convenience. Let $\dis{u\in J_\rh^{-L}}$. By Proposition \ref{mt-impr} it is easy to show that \eqref{al} holds true. Let $\Pi:\{ \dkr(\cdot, M)<\e \}\to M$ be the projection. Set
\begin{equation} \label{s}
s(d_1,d_2)=F\left( \frac{d_1}{d_1+d_2} \right),
\end{equation}
where
$$
	F(x)=\left\{
			\begin{array}{ll}
			0 & \mbox{for } x\in [0,1/4], \\
			2x-\frac 12 & \mbox{for } x\in(1/4\,,3/4), \\
			1 & \mbox{for } x\in[3/4,1].
			\end{array}
	\right. 
$$
$\ov\Psi$ is then defined as
\begin{equation} \label{map}
	\ov\Psi (u) = \left( \Pi\left(\frac{\
h_1e^{u}}{\int_M
h_1e^{u}\,dV_g}\right), \Pi\left(\frac{\
h_2e^{-au}}{\int_M
h_2e^{-au}\,dV_g}\right), s \right).
\end{equation}
\end{proof}

\medskip

\subsection{Improved Moser-Trudinger inequality}

In this subsection we will introduce the ingredients to derive a new improved Moser-Trudinger inequality: the latter one will give some extra constrains to the map $\ov \Psi$ in Proposition \ref{p:altern}. Such inequality is based both on the point of concentration of a function and on the \emph{scale} of concentration. This concepts where introduced in \cite{mal-ru} for treating a Toda system and then in \cite{jev} for the sinh-Gordon case \eqref{eq:sinh} (see also \cite{jkm}). Differently from the latter references where the topological cone on $M$ is used, we will rephrase the argument in terms of the topological join $M*M$. Recall the definition of $\dkr$ in \eqref{dist} and that $M_1\cong M$. Let $\d>0$ and consider the set
$$
	A_\d=\left\{ f\in L^1(M)\,:\, f>0 \mbox{ a.e. and } \int_M f\,dV_g = 1 \right\}\cap\bigr\{\dkr(\cdot,M)<\d\bigr\}
$$
of normalized functions concentrated around a point of the surface. What we need is the following map introduced in \cite{mal-ru}.
\begin{pro}\label{conc}
Let $R>1$ be fixed. Then there exist \,\mbox{$\delta=\delta(R)\!>\!0$} and a continuous map:
\[ \psi : A_\d \rightarrow M\times(0,+\infty), \qquad  \psi(f)= (\beta, \sigma), \]
such that for any $f \in A_\d$ there exists
$p \in \Sigma$ such that
\begin{enumerate}

\item[(1)] $ d(p, \beta) \leq  C \sigma$ with $C = C(R,\d,M).$

\item[(2)] It holds: 
$$ 
\int_{B_\s(p)} f \, dV_g > \tau, \qquad  \int_{B_{R \sigma}(p)^c} f \, dV_g > \tau, 
$$
with $\tau>0$, $\t=\t(R,M)$.

\end{enumerate}
\end{pro}

\begin{proof}
The argument is carried out as in \cite{mal-ru}. We sketch the main steps for the reader's convenience. We start by taking $R_0=3R$ and setting $ \sigma: M \times A_\d\rightarrow (0,+\infty)$ such that:
\begin{equation}\label{def-sigma} 
\int_{B_{\sigma(x,f)}(x)} f \, dV_g = \int_{B_{R_0\sigma(x,f)}(x)^c} f \, dV_g. 
\end{equation}
We point out that it holds
\begin{equation} \label{dist1} 
d(x,y) \leq R_0 \max \bigr\{ \sigma(x,f), \sigma(y,f)\} +\min \{ \sigma(x,f), \sigma(y,f) \bigr\}.
\end{equation}
We then define the mapping $ T: M \times A_\d \rightarrow \mathbb{R}$ by
$$
 T(x,f) = \int_{B_{\sigma(x,f)}(x)} f \,dV_g.
$$
It is possible to show that if $x_0 \in M$ is such that $T(x_0,f) = \max_{y \in M} T(y,f)$, then we have 
\begin{equation} \label{sigma}
\sigma(x_0,f) < 3\,\sigma(x, f), \qquad \forall x \in M, x\neq x_0.
\end{equation}
Exploiting the latter fact, by a covering argument one can prove that there exists a fixed $\tau > 0$ such that
\begin{equation} \label{tau}
  \max_{x \in M} T(x,f) > \tau  \qquad \forall f \in A_\d.
\end{equation}
We define now a continuous function $\sigma : A_\d \rightarrow \mathbb{R}$
$$ 
\sigma(f)= 3 \min\bigr\{ \sigma(x,f): \ x \in A_\d \bigr\}.
$$ 
Let $\tau$ be as in \eqref{tau} and let
\begin{equation} \label{defS}
  S(f) = \bigr\{ x \in M \; : \; T(x,f) > \tau,\ \sigma(x,f) <  \sigma(f) \bigr\},
\end{equation}
Observe that by \eqref{sigma}, \eqref{tau} the latter set is a nonempty open set for any $f\in A_\d$. Moreover, from \eqref{dist1} we have that
\begin{equation}\label{diamS} 
diam\bigr(S(f)\bigr) \leq (R_0+1)\sigma(f). 
\end{equation}
Consider now an embedding of $M$ into $\mathbb{R}^3$ and identify $M$ with its image through the latter embedding. We consider a sort of center of mass $m(f)\in \mathbb{R}^3$
$$
  m(f) = \frac{\displaystyle \int_M \bigr(T(x,f) - \tau\bigr)^+ \bigr( \sigma(f) - \sigma(x,f) \bigr)^+ x  \,dV_g}{\displaystyle \int_M \bigr(T(x,f) - \tau\bigr)^+ \bigr( \sigma(f) - \sigma(x,f) \bigr)^+ \,dV_g},
$$
where $x^+=\max\{x,0\}$. Observe that the integrands become nonzero only on the set $S(f)$. We point out that for a $f\in A_\d$ we have $\dkr(f,M)<\d$ and hence most of its volume is accumulated in a ball (depending on $\d$) centered in a point of the surface. Therefore, by definition there exists a $\wtilde \d=\wtilde \d(\d)$ such that  $\sigma(f) \leq \wtilde\delta$. For $\wtilde\d>0$ sufficiently small, whenever $\sigma(f) \leq \wtilde\delta$ \eqref{diamS} implies that $m(f)$ is close to $M$. Letting $P$ be a orthogonal projection from a $\wtilde\d$-neighborhood of $M$ onto the surface we define $\beta: A_\d \rightarrow M$ to be
$$
\beta(f)= P \circ m (f).
$$
Then the map $\psi(f)=\bigr(\beta(f), \sigma(f)\bigr)$ satisfies the desired properties: indeed, we have just to observe that by \eqref{diamS} we have $d(\b(f),S(f))\leq (R_0+1)\sigma(f)$ and that $\s(f)\leq 3\s(x,f)\leq 3\s(f)$ (recall that $R_0=3R$).
\end{proof}

\medskip

The idea is that the latter map $\psi(f)=(\b,\s)$ gives us the point of concentration of $f$ and its scale of concentration around it. The smaller is $\s$ the faster is the concentration of $f$. We state now the main result of this subsection, i.e. the improved Moser-Trudinger inequality: roughly speaking, whenever the two functions $e^u, e^{-au}$ concentrate both at the same point and at the same scale of concentration we derive an improved inequality.
\begin{pro} \label{mt-improved}
Let $\varepsilon>0$. Then, there exist $R=R(\varepsilon)>1$ and $\psi$ as given in Proposition \ref{conc}, such that for any $u \in H^1(M)$ such that
$$\psi \left( \frac{e^{u}}{\int_{M} e^{u} \,dV_g} \right
)= \psi \left( \frac{e^{-au}}{\int_{M} e^{-au} \,dV_g} \right
),
$$
there exists some $C=C(\varepsilon)$ such that, if $a \geq \dfrac 12$
\begin{equation} \label{imp-ineq}
16\pi \log \int_M e^{u-\ov{u}} \,dV_g +\frac{16\pi}{a^2}\log \int_M e^{-a(u-\ov{u})} \,dV_g \leq \frac{1+\e}{2}\int_M |\nabla u|^2 \,dV_g + C,
\end{equation}
and if $a<\dfrac 12$
\begin{equation} \label{imp-ineq2}
16\pi \log \int_M e^{u-\ov{u}} \,dV_g +\left(\frac{8\pi}{a^2}+\frac{16\pi}{a}\right)\log \int_M e^{-a(u-\ov{u})} \,dV_g \leq \frac{1+\e}{2}\int_M |\nabla u|^2 \,dV_g + C.
\end{equation}
\end{pro}

\medskip

The latter result is mainly based on two local versions of the Moser-Trudinger inequality: one in small balls and the other one in annuli with small internal radius which we prove below. It is inspired by \cite{mal-ru} (see also \cite{jev}). In doing this the following lemma concerning  Poincar\'e-Wirtinger and trace inequalities will be used.
\begin{lem} \label{lem:media}
There exists $C>0$ such that for any $u\in H^1(M)$, $p\in M$ and $r>0$ it holds
$$
	\left| \fint_{B_r(p)} u\,dV_g - \fint_{\p B_r(p)} u\,dS_g \right| \leq C\left( \int_{B_r(p)} |\n u|^2\,dV_g \right)^{1/2}\,.
$$
Furthermore, for any $\d\in(0,1)$ there exists $C=C_\d$ such that for any $u\in H^1(M)$, $p\in M$ and $r>0$ one has
$$
	\left| \fint_{B_{\d r}(p)} u\,dV_g - \fint_{ B_r(p)} u\,dV_g \right| \leq C\left( \int_{B_r(p)} |\n u|^2\,dV_g \right)^{1/2}\,.
$$
\end{lem}

\medskip

We start now with the following result which is obtained by a dilation argument.
\begin{lem} \label{ball} 
For any $\varepsilon>0$ there exists $C=C(\varepsilon)>0$ such that
\begin{align} \label{ineq-ball}
\begin{split}
8\pi\log \int_{B_{s/2}(p)} e^{u}\,dV_g + \frac{8\pi}{a^2}\log \int_{B_{s/2}(p)} e^{-au}\,dV_g  \leq & \frac{1}{2}\int_{B_s(p)}|\nabla u|^2\,dV_g + \e\int_{M}|\nabla u|^2\,dV_g \\ &+8\pi\left(1-\frac{1}{a}\right)\ov u(s)+16\pi\left(1+\frac{1}{a^2}\right)\log s+ C,
\end{split}
\end{align}
for any $u \in H^1(M), \ p \in M$, $s>0$ sufficiently small, where $\dis{\ov u(s)= \fint_{B_s(p)} u \,dV_g}$. Moreover, under the same assumptions, if $a<\dfrac 12$ it holds
\begin{align} \label{ineq-ball2}
\begin{split}
8\pi\log \int_{B_{s/2}(p)} e^{u}\,dV_g + \frac{16\pi}{a}\log \int_{B_{s/2}(p)} e^{-au}\,dV_g  \leq & \frac{1}{2}\int_{B_s(p)}|\nabla u|^2\,dV_g + \e\int_{M}|\nabla u|^2\,dV_g \\ &-8\pi\ov u(s)+16\pi\left(1+\frac{2}{a}\right)\log s+ C.
\end{split}
\end{align}
\end{lem}
\begin{proof}
Suppose for simplicity the metric around the point $p$ is flat. Consider then a dilation of $u$ given by
$$ 
v(x)= u(s x+ p). 
$$
We clearly have
$$ 
\int_{B_s(p)}|\nabla u|^2 \,dV_g = \int_{B_1(0)}|\nabla v|^2 \,dV_g, \qquad \int_{B_{s/2}(p)}e^u \,dV_g = s^2 \int_{B_{1/2}(0)}e^v \,dV_g, \qquad \ov u(s)= \fint_{B_1(0)} v\,dV_g. 
$$
Taking into account the above equalities and by applying the local version of the Moser-Trudinger inequality stated in Proposition \ref{mt-local} to the functions $v, -av$ we get inequality \eqref{ineq-ball}. When $a<\dfrac 12$ one observes that $\dfrac{16\pi}{a} < \dfrac{8\pi}{a^2}$ and hence we may apply Proposition \ref{mt-local} to $v, -av$ with the constant $\dfrac{16\pi}{a}$ replacing $\dfrac{8\pi}{a^2}$ to deduce \eqref{ineq-ball2}.
\end{proof}

\medskip

We next consider an annulus and derive an improved inequality by exploiting the Kelvin's transform. 
\begin{lem} \label{mt-annulus} 
Given $\varepsilon>0$, there exists $r_0>0$, $r_0=r_0(\e,M)$ such that for any $r\!\in \!(0,r_0)$ fixed, there exists \mbox{$C\!=C(r,\varepsilon)>\!0$} such that, for any $u\in H^1(M)$ with $u=0$ on $\partial B_{2r}(p)$, if $a\geq \dfrac 12$
\begin{align} \label{ineq-ann}
\begin{split}
8\pi\log \int_{A_p(s, r)}\!\!\! e^{u} \,dV_g + \frac{8\pi}{a^2}\log \int_{A_p(s,r)}\!\!\! e^{-au} \,dV_g \leq &  \frac{1}{2}\int_{A_p(s/2,2 r)}|\nabla u|^2 \,dV_g +\e \int_M |\nabla u|^2 \,dV_g \\
-& 8\pi\left(1-\frac{1}{a}\right)(1+\e)\ov u(s)-16\pi\left(1+\frac{1}{a^2}\right)(1+\e)\log s+ C, 
\end{split}
\end{align}
and if $a<\dfrac 12$
\begin{align} \label{ineq-ann2}
\begin{split}
8\pi\log \int_{A_p(s, r)} e^{u} \,dV_g + \frac{8\pi}{a^2}\log \int_{A_p(s,r)} e^{-au} \,dV_g \leq &  \frac{1}{2}\int_{A_p(s/2,2 r)}|\nabla u|^2 \,dV_g +\e \int_M |\nabla u|^2 \,dV_g \\
+&8\pi(1+\e)\ov u(s)-16\pi\left(1+\frac{2}{a}\right)(1+\e)\log s+ C, 
\end{split}
\end{align}
with $ p \in M$, $s\in(0,r)$, where $\dis{\ov u(s)= \fint_{B_s(p)} u \,dV_g}$.
\end{lem}

\begin{proof}
Suppose for simplicity the metric around the point $p$ is flat. We introduce the Kelvin's transform $ K : A_p(s/2, 2r) \rightarrow A_p(s/2, 2r)$ defined by
$$ 
K(x)= p+ r s \frac{x-p}{\ |x-p|^2}.
$$
$K$ is constructed in such a way that it maps the interior boundary of $A_p(s/2, 2r)$ onto the exterior one and vice versa. By means of the latter map we consider $\wtilde{u}\in H^1(B_p(2r))$ given by
$$
\wtilde{u}(x)= \left \{ \begin{array}{ll} 
u\bigr(K(x)\bigr) -\b \log|x-p|  & \mbox{ for } |x-p| \geq s/2, \\
-\b \log \left(\frac s2\right) & \mbox{ for } |x-p| \leq s/2, 
\end{array} \right. 
$$
where $\b\in\left[ -\frac 4a, 4 \right]$ will be chosen later. We aim to apply the local Moser-Trudinger inequality given by Proposition \ref{mt-local} to $\wtilde{u}, -a\wtilde u$. To this end, we need to consider
\begin{align} \label{exp}
\begin{split} 
\int_{A_p(s, r)} e^{\wtilde{u}} \,dV_g &=  \int_{A_p(s,r)} e^{u(K(x))}|x-p|^{-\b} \,dV_g = \int_{A_p(s,r)} e^{u(K(x))}\frac{|x-p|^{4-\b}}{s^2r^2}\frac{(sr)^2}{|x-p|^4} \,dV_g \\
&= \int_{A_p(s,r)} e^{u(K(x))}\frac{(sr)^{2-\b}}{|K(x)-p|^{4-\b}}|J(K(x))| \,dV_g = \int_{A_p(s, r)} e^{u(x)} \frac{(sr)^{2-\b}}{|x-p|^{4-\b}} \,dV_g, \\
&\geq C(r)\int_{A_p(s, r)} e^{u(x)} s^{2-\b} \,dV_g,
\end{split}
\end{align} 
where $J(K(x))$ denotes the Jacobian of $K$; in the last inequality we have used $|x|<r$ and $\b\leq 4$. Reasoning in a similar way we get
\begin{equation}\label{exp2}
\int_{A_p(s, r)} e^{\wtilde{u}} \,dV_g = C(r)\int_{A_p(s, r)} e^{-au(x)} s^{2+ a\b} \,dV_g,
\end{equation}
where we used $\b\geq -\frac 4a$. Therefore, by \eqref{exp}, \eqref{exp2} and then by Proposition \ref{mt-local} applied to $\wtilde{u}, -a\wtilde u$ in $B_{2r}(p)$ we obtain 
\begin{align*} 
	8\pi\log \int_{A_p(s, r)}\!\! e^{u(x)} \,dV_g + \frac{8\pi}{a^2}\log \int_{A_p(s, r)} \!\!e^{-au(x)} \,dV_g =& 	8\pi\log \int_{A_p(s, r)}\!\! e^{\wtilde u(x)} \,dV_g + \frac{8\pi}{a^2}\log \int_{A_p(s, r)} \!\!e^{-a	\wtilde u(x)} \,dV_g \\
		&+\left( 8\pi(\b-2)-\frac{8\pi}{a^2}(a\b+2) \right)\log s + C(r) 
\end{align*}
\begin{align} \label{stim}
\begin{split}
\leq  \frac 12 \int_{B_{2r}(p)} |\n \wtilde u(x)|^2\,dV_g + 8\pi\left( 1-\frac 1a \right)\ov{\wtilde u}(2r)+\left( 8\pi(\b-2)-\frac{8\pi}{a^2}(a\b+2) \right)\log s + C(r),
\end{split}
\end{align}
where $\ov{\wtilde u}(2r)=\fint_{B_{2r}(p)} \wtilde u \,dV_g$. To estimate the average part we use Lemma \ref{lem:media} to deduce
$$
	\left|\ov{\wtilde u}(2r) -\fint_{\p B_{2r}(p)} \wtilde u \,dS_g\right| \leq C\left( \int_{B_{2r}(p)} |\n \wtilde u(x)|^2\,dV_g \right)^{1/2} \leq \e \int_{B_{2r}(p)} |\n \wtilde u(x)|^2\,dV_g + C.
$$
Moreover, it holds 
$$
	\fint_{\p B_r(p)} \wtilde u \,dS_g = \fint_{\p B_r(s)} u \,dS_g + C(r).
$$
Therefore, still by Lemma \ref{lem:media} we have
\begin{equation}\label{media}
\left|\ov{\wtilde u}(2r) -\ov{u}(s)\right| \leq \e \int_{B_{2r}(p)} |\n \wtilde u(x)|^2\,dV_g +\e \int_{B_{s}(p)} |\n u(x)|^2\,dV_g + C.
\end{equation}
We are left with the estimate of the gradient term. Since $\wtilde u$ is constant for $|x-p|< s/2$ we need just to consider $|x-p|\geq s/2$, where we have
$$
|\nabla \wtilde{u}(x)|^2 = |\nabla u(K(x))|^2 \frac{s^2 r^2}{|x-p|^4}+\frac{\b^2}{|x-p|^2}+2\b \nabla u(K(x))\cdot \frac{x-p}{|x-p|}sr=G_1+G_2+G_3. 
$$
It is easy to see that
\begin{equation} \label{g1}
	\int_{A_p(s/2, 2r)} G_1 \,dV_g=	\int_{A_p(s/2, 2r)} |\n u|^2 \,dV_g ,
\end{equation}
\begin{equation}\label{g2}
	\int_{A_p(s/2, 2r)} G_2 \,dV_g=	-2\pi\b^2\log s +C(r).
\end{equation}
Now, by using the definition of $K$, by integrating by parts and by $u=0$ on $\partial B_{2r}(p)$ we get
\begin{align*}
	\int_{A_p(s/2, 2r)} G_3 \,dV_g&=	2\b\int_{A_p(s/2, 2r)} \n u(K(x))\cdot \frac{K(x)-p}{|K(x)-p|^2}|J(K(x))|\,dV_g \\
	&= 2\b\int_{A_p(s/2, 2r)} \n u(x) \cdot\frac{x-p}{|x-p|^2}\,dV_g= -2\b\int_{\p B_{s/2}(p)}  u(x) \frac{x-p}{|x-p|^2}\cdot \nu \,dS_g,
\end{align*}
where $\nu$ is the unit outer normal. Observe that $\int_{\p B_{s/2}(p)}\frac{x-p}{|x-p|^2}\cdot \nu \,dS_g=2\pi$ and $\left|\frac{x-p}{|x-p|^2}\cdot \nu\right|\leq \frac{C}{|B_{s/2}(p)|}$. Therefore,
\begin{align*}
\left| \int_{\p B_{s/2}(p)}\!\!\!  u(x) \frac{x-p}{|x-p|^2}\cdot \nu \,dS_g - 2\pi\fint_{\p B_{s/2}(p)} \!\!\! u(x) \,dS_g\right| &= \left| \int_{\p B_{s/2}(p)}\frac{x-p}{|x-p|^2}\cdot \nu\left(  u(x)-\fint_{\p B_{s/2}(p)} \!\!\! u(x) \,dS_g\right)\right| \\
	&\leq C \left| \fint_{\p B_{s/2}(p)}\left(  u(x)-\fint_{\p B_{s/2}(p)}  u(x) \,dS_g\right)\right| \\
	&\leq C \left(\int_{B_{s/2}(p)} |\n u|^2 \,dV_g\right)^{1/2} \\
	&\leq \e \int_{B_{s/2}(p)} |\n u|^2 \,dV_g + C,
\end{align*}
where we have used Lemma \ref{lem:media}. Applying again the latter lemma we deduce that
\begin{equation}\label{g3}
\int_{A_p(s/2, 2r)} G_3 \,dV_g= -2\b\int_{\p B_{s/2}(p)}  u(x) \frac{x-p}{|x-p|^2}\cdot \nu \,dS_g= -4\pi\b\ov u(s) +\e \int_{B_{s/2}(p)} |\n u|^2 \,dV_g + C.
\end{equation}
Finally, by using \eqref{media}, \eqref{g1}, \eqref{g2} and \eqref{g3} in \eqref{stim} we obtain
\begin{align*}
8\pi\log \int_{A_p(s, r)}\!\! e^{u(x)} \,dV_g + \frac{8\pi}{a^2}\log \int_{A_p(s, r)} \!\!e^{-au(x)} \,dV_g &\leq \frac 12 	\int_{A_p(s/2, 2r)} |\n u|^2 \,dV_g+ \left(8\pi\left( 1-\frac 1a \right)-2\pi\b\right)\ov{u}(s) \\
 	&+\left( 8\pi(\b-2)-\frac{8\pi}{a^2}(a\b+2)-\pi\b^2 \right)\log s \\
 	&+ \e\int_M |\n u|^2 \,dV_g + \e \int_{B_{2r}(p)} |\n \wtilde u(x)|^2\,dV_g + C.
\end{align*}
To conclude we just need to take either $\b=8\left( 1-\frac 1a \right)$ for $a\geq \frac 12$ or $\b=-\frac 4a$ for $a<\frac 12$ to get the desired inequalities \eqref{ineq-ann} and \eqref{ineq-ann2}, respectively.
\end{proof}

\medskip

Now we have all the ingredients to prove the main Proposition \ref{mt-improved}.

\medskip

\begin{proof}[Proof of Proposition \ref{mt-improved}.] 
The strategy follows the same steps as in the proof of Proposition 3.2 in \cite{mal-ru} or Proposition 3.6 in \cite{jev} hence we will just present here the main idea.

Let $\psi$ be the map defined in Proposition \ref{conc}. Let $u\in H^1(M)$ be such that 
$$
\psi \left( \frac{e^{u}}{\int_{M} e^{u} \,dV_g} \right
)= \psi \left( \frac{e^{-au}}{\int_{M} e^{-au} \,dV_g} \right
) = (\b,\s).
$$
Then, by Proposition \ref{conc} there exist $p_1, p_2\in M$, $d(p_1,p_2)\leq C\s$ such that
\begin{align*}
&\int_{B_\s(p_1)} e^u \, dV_g > \tau \int_M e^u \,dV_g, \qquad \qquad \int_{B_{R \sigma}(p_1)^c} e^u \, dV_g > \tau\int_M e^u \,dV_g,  \\
&\int_{B_\s(p_2)} e^{-au} \, dV_g > \tau \int_M e^{-au} \,dV_g, \qquad \int_{B_{R \sigma}(p_2)^c} e^{-au} \, dV_g > \tau\int_M e^{-au} \,dV_g, 
\end{align*}
with $\tau>0$ independent of $\s$. Suppose for a moment that $p_1=p_2$. Then, we may apply Lemma \ref{ball}, Lemma \ref{mt-annulus}: summing the inequalities \eqref{ineq-ball} and \eqref{ineq-ann} if $a\geq \frac 12$ (resp. \eqref{ineq-ball2} and \eqref{ineq-ann2} if $a< \frac 12$) the extra term 
$$
8\pi\left(1-\frac{1}{a}\right)\ov u(\s)+16\pi\left(1+\frac{1}{a^2}\right)\log \s, \quad  a\geq \frac 12 \qquad \left(\mbox{resp.} -8\pi\ov u(\s)+16\pi\left(1+\frac{2}{a}\right)\log \s , \quad a<\frac 12\right)
$$
cancels out and we get the desired inequality of Proposition \ref{mt-improved}. However, one needs to face the fact that in general $p_1\neq p_2$ and that $u$ is not identically zero on some $\partial B_{2r}(p)$ as in Lemma \ref{mt-annulus}. To deal with these facts we have to perform some technical modifications involving  dyadic decompositions and harmonic liftings: for full details we refer to \cite{jev, mal-ru}. 
\end{proof}

\medskip

\subsection{Topological set and test functions}
In this subsection we will introduce the topological set which will describe the sublevels $J_\rho^{-L}$: starting from the topological join $M*M$ (recall that $M_1 \cong M$) according to the constraints imposed by the improved Moser-Trudinger inequality of Proposition \ref{mt-improved}. Next, we will construct test functions modeled on this set.

We need to take into account the local scale $\s$ of functions as defined in Proposition \ref{conc}. Since the latter is defined just for functions $f$ such that $\dkr (f,M)<\d$, for $\d=\d_R>0$, we proceed in the following way. Let
$$
	\s_M=\inf \left\{ \s(f) \,:\,  \dkr (f,M)\leq \frac 12\d \right\},
$$ 
and set 
$$
	\ov\s(u) = \min \left\{ \s_M,  \s\left( \frac{e^{u}}{\int_{M} e^{u} \,dV_g}\right) \right\}.
$$
Whenever $\s\left( \frac{e^{u}}{\int_{M} e^{u} \,dV_g}\right)$ is not well-defined we will consider $\s_M$. Let $\psi$ be as in Proposition \ref{conc} and let $s$ be as defined in \eqref{s}. We use the notation
\begin{equation} \label{not}
 \psi \left( \frac{e^{u}}{\int_{M} e^{u} \,dV_g}\right) = \bigr( \b(u), \ov\s(u)  \bigr).
\end{equation}
In the spirit of \eqref{map} we then set $\wtilde\Psi : J_\rho^{-L} \to M*M$,
\begin{equation} \label{Psi}
	\wtilde\Psi(u) = \biggr( \b(u), \b(-au), s\bigr((\ov\s(u),\ov\s(-au)\bigr) \biggr).
\end{equation}
Observe that the point of concentration $\beta$ is defined just for functions $f$ such that $\dkr (f,M)<\d$. However, notice that by the Proposition \ref{p:altern} when one of $\b(u), \b(-au)$ is not defined the other necessarily is, and the map is well defined by the equivalence relation, see \eqref{join}, for $L>0$ sufficiently large. Indeed, we have just to observe that $\ov\s(f)\approx\dkr (f,M)$.  Moreover, the improved inequality of Proposition \ref{mt-improved} gives a lower bound on configurations which have both the same point of concentration $\b$ and scale of concentration $\ov\s$: such configurations are represented by
\begin{equation} \label{S}
 S=\left\{ \left(p,p,\frac 12\right)\,:\, p\in M  \right\} \subset M*M.
\end{equation}
Therefore, we deduce the following result.
\begin{pro} \label{proj}
Suppose $\rho_1 \in (8\pi, 16\pi)$ and either $\rho_2 \in \left(\dfrac{8\pi}{a^2}, \dfrac{16\pi}{a^2} \right)$ if $a\geq \dfrac 12$ or $\rho_2 \in \left(\dfrac{8\pi}{a^2}, \dfrac{8\pi}{a^2}+\dfrac{16\pi}{a} \right)$ if $a< \dfrac 12$. Let $\wtilde\Psi$ be as in \eqref{Psi} and let $S$ be as in \eqref{S}. Then, for $L$ sufficiently large it holds that
$$
	\wtilde\Psi : J_\rho^{-L} \to (M*M) \setminus S.
$$
\end{pro}

Moreover, we can construct a map on the other way round by mapping $(M*M) \setminus S$ into the sublevels $J_\rho^{-L}$. More precisely, we will consider a deformation retract $X_\l$ of $(M*M) \setminus S$ which is more suitable to modeled the test functions on. Indeed, let $\bar\d>0$ be sufficiently small and $\l>0$ be sufficiently large: in $(M*M) \setminus S$ we can either deform two distinct points up to have mutual distance at least $\bar\d$ or deform the join parameter $s$ to be very close either to $0$, i.e. $s\leq \frac 1\l$ (when $s<1/2$), or to $1$, i.e. $s\geq 1-\frac 1\l$ (when $s>1/2$). In doing this we end up with the set $X_\l\subset (M*M) \setminus S$ defined as
\begin{equation} \label{x}
	X_\l = \biggr\{ \bigr(p,q,s\bigr)\,:\, p,q\in M, \, d(p,q)\geq\bar\d, \, s\in[0,1] \biggr\} \cup \left\{ \bigr(p,q,s\bigr)\,:\, p,q\in M, \, d(p,q)\leq\bar\d, \, s\leq\frac 1\l \  \mbox{or} \ s\geq 1-\frac 1\l \right\},
\end{equation}
with the convention that whenever $s\in\{0,1\}$ we do not impose any restriction on the points $p,q$ (recall the equivalence relation in the definition of the topological join, see \eqref{join}). We consider now test functions modeled on the latter set: for $\xi=\bigr(p,q,s\bigr)\in X_\l$ we define
$$
	\l_{1,s}=(1-s)\l, \qquad \l_{2,s}=s\l,
$$
and 
\begin{equation} \label{test-f}
	\wtilde\Phi(\xi)= \var_{\l,\xi}(x)=  \log  \left( \frac{1}{1 + \l_{1,s}^2 d(x,p)^2} \right)^2- \frac 1a \log \left( \frac{1}{1 + \l_{2,s}^2 d(x,q)^2} \right)^2.
\end{equation}
Observe that the above map is well defined in the topological join due to the expressions of $\l_{i,s}$. In the following result we will show that $\wtilde\Phi$ maps $X_\l$ into the low sublevels $J_\rho^{-L}$.
\begin{pro} \label{p:test2}
Suppose $\rho_1 > 8\pi$ and $\rho_2 > \dfrac{8\pi}{a^2}$ with $a\in(0,1)$. Let $X_\l$ and $\var_{\l,\xi}$ be as given in \eqref{x}, \eqref{test-f}, respectively. Then, it holds
$$
	J_\rho \left( \var_{\l,\xi} \right) \to -\infty \quad \mbox{as } \l\to+\infty \qquad \mbox{uniformly in } \xi\in X_\l.
$$
Letting $\wtilde\Phi$ be defined as in \eqref{test-f}, it follows that for any $L>0$ there exists $\l>0$ large such that
$$
	\wtilde\Phi: X_\l \to J_\rho^{-L}.
$$
\end{pro}

\begin{proof}
Let $v_1,v_2: M \rightarrow \mathbb{R}$ be given by
\begin{equation} \label{v}
    v_1(x) = \log  \left( \frac{1}{1 + \l_{1,s}^2 d(x,p)^2} \right)^2, \qquad
    v_2(x) = \log \left( \frac{1}{1 + \l_{2,s}^2 d(x,q)^2} \right)^2.
\end{equation}
so that $\var_{\l,\xi}=v_1-\frac 1a v_2$. Observe that by construction of the set $X_\l$ in \eqref{x} we need to carry out the energy estimates in the following two regimes: either the two points of concentration $p,q\in M$ are close and the scale of concentration are very different, i.e. $d(p,q)\leq \bar\d$ and $s=1-\frac 1\l$ (resp. $s=\frac 1\l$), or $p,q$ are such that $d(p,q)\geq \bar\d$. We start by pointing out that in the first alternative we have $\l_{1,s} \leq 1$ (resp. $\l_{2,s}\leq 1$), see the definition before \eqref{test-f}. It follows that $v_1$ (resp. $v_2$) and its derivative are uniformly bounded. Therefore, the test function $\var_{\l,\xi}$ resembles the standard bubble or the one in \eqref{bubble} for which the energy estimates are well known, see Proposition \ref{p:test}.

Let us consider now the case $d(p,q)\geq \bar\d$. Moreover, we take $s\in(0,1)$ otherwise we conclude as before. As in the proof of Proposition \ref{p:test} it holds
\begin{equation} \label{gr1-}
    |\n v_i(x)| \leq C \l_{i,s}, \qquad \mbox{for every $x\in M$ and $s\in[0,1],$} \quad i=1,2,
\end{equation}
where $C$ is a constant independent of $\l$, $\xi \in X_\l$, and
\begin{equation} \label{gr2-}
    |\n v_1(x)| \leq \frac{4}{d(x,p)}, \qquad \mbox{for every $x\in M,$} \quad i=1,2,
\end{equation}
and similarly for $v_2$.

\medskip

We have
$$
    \frac 12 \int_{M} |\n \var_{\l,\xi}|^2 \,dV_g  =  \frac{1}{2} \int_M \left(|\n v_1|^2 + \frac{1}{a^2}|\n v_2|^2 - \frac{2}{a} \n v_1 \cdot \n v_2\right) \,dV_g. 
$$
Reasoning as in Proposition 3.3 in \cite{bjmr} it is easy to show that the integral of the mixed term $\n v_1 \cdot\n v_2$ is bounded by a constant depending only on $M$, i.e.
\begin{equation} \label{eq:misto}
\int_M \n v_1 \cdot\n v_2 \,dV_g \leq C.
\end{equation}

Exploiting the fact that $d(p,q)\geq \bar\d$ and using the estimates \eqref{gr1-}, \eqref{gr2-}, we can proceed as in the proof of Proposition \ref{p:test}, see \eqref{grad}, to deduce
\begin{equation} \label{gr}
    \frac 12 \int_{M} |\n \var_{\l,\xi}|^2 \,dV_g \leq 16 \pi\bigr(1 + o_\l(1)\bigr) \log \bigr(\l_{1,s} + \d_{1,s}\bigr) + \frac{16\pi}{a^2} \bigr(1 + o_\l(1)\bigr) \log \bigr(\l_{2,s} + \d_{2,s}\bigr) + C,
\end{equation}
where $\d_{1,s} > \d > 0$ as $s \to 1$ and $\d_{2,s} > \d > 0$ as $s \to 0$, for
some fixed $\d$. 

\medskip

The same argument as in the proof of Proposition \ref{p:test}, see \eqref{aver}, leads to 
$$
  \int_M v_1 \,dV_g = - 4 \bigr(1 + o_\l(1)\bigr) \log \bigr(\l_{1,s} + \d_{1,s}\bigr) + O(1); \qquad \int_M v_2 \,dV_g = - 4 \bigr(1 + o_\l(1)\bigr) \log \bigr(\l_{2,s} + \d_{2,s}\bigr) + O(1),
$$
therefore we obtain
\begin{eqnarray}
    \int_M \var_{\l,\xi} \,dV_g = - 4 \bigr(1 + o_\l(1)\bigr) \log \bigr(\l_{1,s} + \d_{1,s}\bigr) + \frac 4a \bigr(1 + o_\l(1)\bigr) \log \bigr(\l_{2,s} + \d_{2,s}\bigr) + O(1). \label{av1} 
\end{eqnarray}

\medskip

We are left with estimating 
$$
    \int_M e^{\var_{\l,\xi}} \,dV_g =  \int_M \frac{1}{\bigr( 1 + \l_{1,s}^2 d(x,p)^2 \bigr)^2} \bigr( 1 + \l_{2,s}^2 d(x,q)^2 \bigr)^{2/a}\,dV_g(x).
$$
We consider $M = B_{\bar\d/2}(p) \cup (M \setminus B_{\bar\d/2}(p))$. In $B_{\bar\d/2}(p)$ we observe that $\frac 1C \leq d(x, q) \leq C$ and hence
$$
\int_{B_{\bar\d/2}(p)} \frac{1}{\bigr( 1 + \l_{1,s}^2 d(x,p)^2 \bigr)^2} \bigr( 1 + \l_{2,s}^2 d(x,q)^2 \bigr)^{2/a}\,dV_g(x) = \frac{\bigr(\l_{2,s} + \d_{2,s}\bigr)^{\frac 4a}}{\bigr(\l_{1,s} + \d_{1,s}\bigr)^2}\bigr(1 + O(1)\bigr).
$$
In $M \setminus
B_{\bar\d/2}(p)$ we have that $\frac 1C \leq d(x, p) \leq C$ and we deduce that this part is a higher-order term. We conclude that
\begin{equation} \label{exp1-}
\log \int_M {h}_1 e^{\var_{\l,\xi}} \,dV_g = \frac 4a \log \bigr(\l_{2,s} + \d_{2,s}\bigr) - 2 \log \bigr(\l_{1,s} + \d_{1,s}\bigr) + O(1).
\end{equation}
Similarly we get
\begin{equation} \label{exp2-}
\log \int_M {h}_2 e^{-a \var_{\l,\xi}} \,dV_g = 4a\log \bigr(\l_{1,s} + \d_{1,s}\bigr) - 2\log \bigr(\l_{2,s} + \d_{2,s}\bigr) + O(1).
\end{equation}

\medskip

Finally, using the expression of $J_\rho$ in \eqref{func} and the estimates (\ref{gr}), (\ref{av1}), (\ref{exp1-}) and (\ref{exp2-}) we assert that
$$
  J_\rho(\var_{\l,\xi}) \leq \bigr( 16 \pi - 2 \rho_1 + o_\l(1) \bigr)\log \bigr(\l_{1,s} + \d_{1,s}\bigr) + \left(  \frac{16\pi}{a^2} - 2 \rho_2 + o_\l(1) \right)\log \bigr(\l_{2,s} + \d_{2,s}\bigr) + O(1).
$$
Observe that $\dis{\max_{s\in[0,1]}\{ \l_{1,s}, \l_{2,s} \}} \to +\infty$ as $\l \to \infty$. The proof is done since $\rho_1 > 8\pi, \rho_2 > \dfrac{8\pi}{a^2}$ by assumption.

\end{proof}

We prove now the existence result of Theorem \ref{th:ex2}.

\medskip

\begin{proof}[Proof of Theorem \ref{th:ex2}.]
Suppose $\rho_1 \in (8\pi, 16\pi)$ and either $\rho_2 \in \left(\dfrac{8\pi}{a^2}, \dfrac{16\pi}{a^2} \right)$ if $a\geq \dfrac 12$ or $\rho_2 \in \left(\dfrac{8\pi}{a^2}, \dfrac{8\pi}{a^2}+\dfrac{16\pi}{a} \right)$ if $a< \dfrac 12$. Let $X_\l$ be as in \eqref{x} and denote by $\mathcal{R}$ the deformation retraction involved in its definition. Let $\wtilde\Psi$ be as in Proposition \ref{proj} and let $\wtilde \Phi, \var_{\l,\xi}$ be as in \eqref{test-f}. The key fact is to show that 
\begin{equation} \label{comp}
	X_\l \stackrel{\wtilde\Phi}{\longrightarrow} J_\rho^{-L} \xrightarrow{\mathcal{R}\circ\wtilde\Psi} X_\l
\end{equation}
is homotopic to {Id}{$_{|X_\l}$} for $\l$ large. Let $\xi=(p,q,s)\in X_\l$. Recalling the notation in \eqref{not} we need to consider
$$
		\wtilde\Psi(\var_{\l,\xi}) = \biggr( \b(\var_{\l,\xi}), \b(-a\var_{\l,\xi}), s\bigr((\ov\s(\var_{\l,\xi}),\ov\s(-a\var_{\l,\xi})\bigr) \biggr).
$$
Recall the definition of $\dkr$ in \eqref{dist}. Reasoning as in Proposition 4.9 in \cite{bjmr} there exist $C>0$ not depending on $\l, s$ such that for any $\xi=(p,q,s)\in X_\l$ with $d(p,q)\geq \bar \d$ and $s\in(0,1)$ we have
\begin{align} \label{d}
\begin{split}
\frac 1C \min\left\{1,\frac{1}{(1-s)\l}\right\} &\leq \dkr \left (\frac{{h}_1 e^{\var_{\l,\xi}}}{\int_{M} {h}_1 e^{\var_{\l,\xi}} \, dV_g}, M \right ) \leq \frac{C}{(1-s)\l}, \\
\frac 1C \min\left\{1,\frac{1}{s\l}\right\} &\leq \dkr \left (\frac{{h}_2 e^{-a\var_{\l,\xi}}}{\int_{M} {h}_2 e^{-a\var_{\l,\xi}} \, dV_g}, M \right ) \leq \frac{C}{s\l}.
\end{split}
\end{align}
Consider now $d(p,q)\leq \bar\d$. By the construction of the set $X_\l$ we readily have one of the two components $v_1, v_2$ defined in \eqref{v} is bounded, i.e. one bubble is negligible, see the argument at the beginning of the proof of Proposition \ref{p:test2}. Therefore, we can still apply the argument in \cite{bjmr} to deduce the above estimates (in this case we will have either $\min\left\{1,\frac{1}{(1-s)\l}\right\}=\frac{1}{(1-s)\l}$ and $\min\left\{1,\frac{1}{s\l}\right\}=1$ or the switched situation).

\medskip

Concerning the scale of concentration $\ov\s$, by estimating the volume of the test functions in small balls as in Lemma 4.4 in \cite{mal-ru} (see also \cite{jev}) it is not difficult to see that there exists $C>0$ not depending on $\l, s$ such that
\begin{equation} \label{sca}
	\frac 1C \leq \dfrac{ \ov\s(\var_{\l,\xi}) }{ \min\left\{1,\dfrac{1}{(1-s)\l}\right\}} \leq C, \qquad \frac 1C \leq \dfrac{\ov\s(-a\var_{\l,\xi})}{\min\left\{1,\dfrac{1}{s\l}\right\}} \leq C,
\end{equation}
see the proof of Lemma 4.5 in \cite{mal-ru} (see also \cite{jev}). Observe that by \eqref{d} when $\min\left\{1,\frac{1}{(1-s)\l}\right\}=1$ we get $\ov\s(\var_{\l,\xi})=\s_M$, see the notation before \eqref{not}, and the above estimate trivially holds true. A similar argument works for $\ov\s(-a\var_{\l,\xi})$.

It follows that when one of the projections $\displaystyle{\Pi\left(\frac{\ h_1e^{\var_{\l,\xi}}}{\int_M h_1e^{\var_{\l,\xi}}\,dV_g}\right), \Pi\left(\frac{\ h_2e^{-a\var_{\l,\xi}}}{\int_M h_2e^{-a\var_{\l,\xi}}\,dV_g}\right)}\in M$ is not defined, where $\Pi$ is given before \eqref{s}, the other necessarily is, and the map is well defined by the equivalence relation, see \eqref{join}, for $\l>0$ sufficiently large. Moreover, $\displaystyle{\Pi\left(\frac{\ h_1e^{\var_{\l,\xi}}}{\int_M h_1e^{\var_{\l,\xi}}\,dV_g}\right)\to p}$ and $\displaystyle{\Pi\left(\frac{\ h_2e^{-a\var_{\l,\xi}}}{\int_M h_2e^{-a\var_{\l,\xi}}\,dV_g}\right)}\to q$ as $\l\to+\infty$ whenever they are well defined, see for example the proof of Theorem \ref{th:ex1}. On the other hand, we clearly have $\displaystyle{\Pi\left(\frac{\ h_1e^{\var_{\l,\xi}}}{\int_M h_1e^{\var_{\l,\xi}}\,dV_g}\right)\approx \b(\var_{\l,\xi})}$ and $\displaystyle{\Pi\left(\frac{\ h_2e^{-a\var_{\l,\xi}}}{\int_M h_2e^{-a\var_{\l,\xi}}\,dV_g}\right)}\approx \b(-a\var_{\l,\xi})$ for $\l$ sufficiently large since all the volume of the test functions is accumulating around $p$ or $q$, respectively, see the argument of Lemma 4.5 in \cite{mal-ru} for an alternative proof of the latter property (see also \cite{jev}). We conclude that $\b(\var_{\l,\xi})\approx p$ and $\b(-a\var_{\l,\xi})\approx q$ for $\l$ large, whenever they are well defined.

\medskip

Therefore, the desired homotopy is obtained by deforming $\b(\var_{\l,\xi})$ to $p$, $\b(-a\var_{\l,\xi})$ to $q$, whenever they are well defined and by deforming $s\bigr((\ov\s(\var_{\l,\xi}),\ov\s(-a\var_{\l,\xi})\bigr)$ to the initial $s$. We have just to check that $\wtilde\Psi(\var_{\l,\xi})\in (M*M)\setminus S$, i.e. that $s\bigr((\ov\s(\var_{\l,\xi}),\ov\s(-a\var_{\l,\xi})\bigr)\neq \frac 12$ for $d(p,q)\leq \bar\d$, see \eqref{S}. Indeed, we have already observed below \eqref{d} that in this case we get either $\min\left\{1,\frac{1}{(1-s)\l}\right\}=\frac{1}{(1-s)\l}$ and $\min\left\{1,\frac{1}{s\l}\right\}=1$ or the switched situation. Suppose the first alternative holds true. Using then \eqref{sca} we conclude that $\ov\s(\var_{\l,\xi})\ll \ov\s(-a\var_{\l,\xi})$ for $\l$ large and hence $s\bigr((\ov\s(\var_{\l,\xi}),\ov\s(-a\var_{\l,\xi})\bigr)\neq \frac 12$ by definition. 

\medskip

This concludes the proof of the fact that the composition in \eqref{comp} is homotopic to {Id}{$_{|X_\l}$} for $\l$ large. Therefore, we deduce that
$$
	H_q(X_\l) \hookrightarrow H_q \left(J_{\rho}^{-L}\right).
$$
We next observe that $X_\l$ has non-trivial homology which leads to non-trivial homology of $J_{\rho}^{-L}$. Since $X_\l$ is a deformation retract of $(M*M)\setminus S$ it is enough to consider the homology of the latter set. We point out that the positive genus case $g(M)>0$ can be treated as in Section \ref{sec:genus}, which yields existence of solutions to \eqref{eq}. Hence, we restrict our attention to $M\cong\S^2$. In this case we get $\S^2*\S^2\cong \S^5$ and $S\cong \S^2$. Therefore, by the Alexander duality, see the Corollary 3.45 in \cite{hat}, we obtain $H_2((\S^2*\S^2)\setminus \S^2)\cong H^2(\S^2)\cong\mathbb Z$.

 Therefore, by applying Proposition \ref{top-arg} and Proposition \ref{high} as in the proof of Theorem~\ref{th:ex1} we get the conclusion.
\end{proof}

\

\section{Existence result in a supercritical case with positive genus} \label{sec:genus}

In this section we are concerned with the supercritical range $\rho_1 \in (8\pi, 16\pi)$ and $\rho_2 \in \left(\dfrac{8\pi}{a^2}+\dfrac{16\pi}{a}, \dfrac{16\pi}{a^2} \right)$ for $a< \dfrac 12$. The goal is to get the existence result of Theorem \ref{th:ex3}. Observe that the improved inequality of Proposition \ref{mt-improved} does not apply to this case and hence the strategy of the previous section does not work out for this range of the parameters. To overcome the difficulties we will restrict ourselves to a surface $M$ with positive genus $g(M)>0$ and apply the argument introduced in \cite{bjmr} and used also in \cite{je-ya}, see the discussion in the Introduction below \eqref{join}.

In the previous sections we have already introduced almost all the ingredients to carry the argument out. We start by recalling that if $J_\rho(u)$ is large negative, then either $e^u$ or $e^{-au}$ (or both) need to be concentrated around a point of the surface and hence there is a continuous map 
\begin{equation} \label{pr}
	\ov \Psi: J_\rho^{-L} \to M *M,
\end{equation}
for $L$ sufficiently large, see Proposition \ref{p:altern}. Next, we need the following topological result concerning positive genus surfaces, see \cite{bjmr}.
\begin{lem} \label{new} Let $M$ be a compact surface not homeomorphic to $\S^2$. Then, there exist two
simple closed curves $\gamma_1, \gamma_2 \subseteq M$
such that

\begin{enumerate}
\item $\gamma_1, \gamma_2$ do not intersect
each other ;

\item there exist two global retractions $R_i: M \to \gamma_i$,
$i=1,2$.

\end{enumerate}

\end{lem}

Let $\gamma_i$ be as in the latter lemma. By means of the above retractions we can restrict the map $\ov \Psi$ in \eqref{pr} to targets in the topological join $\gamma_1*\gamma_2$ only. Indeed, recalling the notation in \eqref{map} we consider 
\begin{equation} \label{pr2}
	\ov\Psi_R (u) = \left( (R_1)_*\,\Pi\left(\frac{\
h_1e^{u}}{\int_M
h_1e^{u}\,dV_g}\right), (R_2)_*\,\Pi\left(\frac{\
h_2e^{-au}}{\int_M
h_2e^{-au}\,dV_g}\right), s \right),
\end{equation}
where $(R_i)_*$ stands for the push-forward of the map $R_i$: we have
$$
\ov\Psi_R:J_\rho^{-L} \to \gamma_1*\gamma_2.
$$
Moreover, as in the previous section we can construct a reverse map. Namely, let $\xi=(p,q,s)\in \gamma_1*\gamma_2$, i.e. $p\in \gamma_1$, $q\in \gamma_2$ and consider test functions 
\begin{equation} \label{test2}
\wtilde\Phi_R(\xi)=\var_{\l,\xi}
\end{equation} 
as given in \eqref{test-f}. Then, the following result holds true. 
\begin{pro}
Suppose $\rho_1 > 8\pi$ and $\rho_2 > \dfrac{8\pi}{a^2}$ with $a\in(0,1)$. Let $\var_{\l,\xi}$ be as given in \eqref{test-f}. Then, it holds
$$
	J_\rho \left( \var_{\l,\xi} \right) \to -\infty \quad \mbox{as } \l\to+\infty \qquad \mbox{uniformly in } \xi\in \gamma_1*\gamma_2.
$$ 
Letting $\wtilde\Phi_R$ be defined as in \eqref{test2}, it follows that for any $L>0$ there exists $\l>0$ large such that
$$
	\wtilde\Phi_R: \gamma_1*\gamma_2 \to J_\rho^{-L}.
$$
\end{pro} 

\begin{proof}
We have just to observe that the components $v_1, v_2$ of the test functions $\var_{\l,\xi}$, as defined in \eqref{v}, are supported in $p\in\gamma_1$ and $q\in\gamma_2$, respectively, where $\gamma_1, \gamma_2$ do not intersect each other by construction. Therefore, there exists $\bar\d>0$ small such that $d(p,q)\geq \bar\delta$ and we can carry out all the estimates as in the proof of the Proposition \ref{p:test2}.   
\end{proof}
 
We prove now the existence result of Theorem \ref{th:ex3}.

\medskip

\begin{proof}[Proof of Theorem \ref{th:ex3}.]
Suppose $\rho_1 \in (8\pi, 16\pi)$ and $\rho_2 \in \left(\dfrac{8\pi}{a^2}+\dfrac{16\pi}{a}, \dfrac{16\pi}{a^2} \right)$ with $a< \dfrac 12$. Let $\gamma_1$, $\gamma_2$ be as in Lemma \ref{new}, let $\ov\Psi_R$ be as in \eqref{pr2} and let $\wtilde\Phi_R$ be as in \eqref{test2}. Reasoning as in the proof of Theorem \ref{th:ex2} we get
$$
	\gamma_1*\gamma_2 \stackrel{\wtilde\Phi_R}{\longrightarrow} J_\rho^{-L} \stackrel{\wtilde\Psi_R}{\longrightarrow} \gamma_1*\gamma_2
$$
is homotopic to Id{$_{|\gamma_1*\gamma_2}$} for $\l$ large. It follows that
$$
	H_q(\gamma_1*\gamma_2) \hookrightarrow H_q \left(J_{\rho}^{-L}\right).
$$
We know that $\gamma_1*\gamma_2$ is homeomorphic to $\mathbb{S}^3$, see for example \cite{bjmr} and hence $J_{\rho}^{-L}$ has non-trivial homology. One can then conclude by applying Proposition \ref{top-arg} and Proposition \ref{high} as in the proof of Theorem~\ref{th:ex1}.
\end{proof}

\

\begin{center}
\textbf{Acknowledgements}
\end{center}

The author would like to thank W. Yang for the discussions concerning the topic of this paper and Prof. A. Malchiodi, Prof. S. Kallel for their helpful comments. \\
The author is supported by PRIN12 project: \emph{Variational and Perturbative Aspects of Nonlinear Differential Problems} and FIRB project:
\emph{Analysis and Beyond}.

\

\end{document}